\documentclass[12pt,reqno]{amsart}
\advance \topmargin by -\headheight
\advance \topmargin by -\headsep
\evensidemargin \oddsidemargin
\newtheorem{theorem}{Theorem}[section]
\newtheorem{lemma}[theorem]{Lemma}
\newtheorem{corollary}[theorem]{Corollary}
\newtheorem{proposition}[theorem]{Proposition}

\usepackage{amsmath, amsthm, amssymb, amsfonts}
 \usepackage{hyperref}
\usepackage{enumerate}
\usepackage{url}
\usepackage{dsfont}
\usepackage{mathrsfs}

\usepackage{bm}

\usepackage{xcolor}

\usepackage{fancyvrb}

\theoremstyle{definition}

\theoremstyle{remark}

\numberwithin{equation}{section}

\DeclareMathOperator{\Mod}{mod}
\newcommand{\mmod}[1]{\;(\Mod{ #1})}

\renewcommand{\geq}{\geqslant}
\renewcommand{\leq}{\leqslant}

\def\eps{\varepsilon}

\def \cC {\mathcal C}

\def \cE {\mathcal E}

\def \cM {\mathcal M}

\def \fS {\mathfrak S}

\def \N {\mathbb{N}}

\def \R {\mathbb{R}}

\def \Z {\mathbb{Z}}

\def\d{{\,{\rm d}}}

\def \Li {\mathrm{Li}}

\newcommand{\n}[1]{\|{#1}\|}
\newcommand\1{\mathds{1}}

\renewcommand\vec{\bm}

\title{Counting solutions to the quadratic determinant equation}

\author{Jonathan Chapman \and Akshat Mudgal}
\address{Mathematics Institute, Zeeman Building, University of Warwick, Coventry CV4 7AL, United Kingdom}
\email{Jonathan.Chapman@warwick.ac.uk}
\address{Mathematics Institute, Zeeman Building, University of Warwick, Coventry CV4 7AL, United Kingdom}
\email{Akshat.Mudgal@warwick.ac.uk}


\begin{document}

\begin{abstract}
Given $h, N \in \mathbb{N}$ satisfying $1 \leqslant h \leqslant N^2$, we prove an asymptotic formula for the number of solutions to the equation $x_1 x_2 - x_3 x_4 = h$ with $x_1, \ldots, x_4 \in [-N,N] \cap \mathbb{Z}$. We use a combination of combinatorial and analytic arguments in physical space along with bounds for Kloosterman sums. Our main result concerns the case when $h = N^2 + O(N)$, wherein we obtain square-root cancellation error terms by bypassing Kloosterman sum bounds and exploiting an additional symmetry available in this setting via Ramanujan sums. This confirms a speculation of Dhanda--Haynes--Prasala in a very general form.
\end{abstract}

\maketitle

\section{Introduction}

Many classical topics in analytic number theory concern counting solutions to $Q(\vec{x}) = h$, where $h$ is some positive integer, $Q \in \mathbb{Z}[x_1, x_2, x_3, x_4]$ is some quadratic form and $\vec{x} = (x_1, \dots, x_4)$ varies over lattice points in some expanding region. This is especially true when $Q$ satisfies $Q(\vec{x}) = x_1 x_4 - x_2 x_3$. For instance, a well known result of Selberg implies that for all $N \in \mathbb{N}$, one has
\begin{equation} \label{eqn1.1}
|\{ \vec{x} \in \mathbb{Z}^4:   Q(\vec{x}) = 1 \ \text{and}  \ x_1^2 + \dots + x_4^2 \leq N^2\}| = 6N^2 + O(N^{4/3}). 
\end{equation}
The error term here hasn't been improved since Selberg's result and it is conjectured the error term should of the order $O_{\eps}(N^{1 + \eps})$, see \cite{Iw2002}. Similarly,  the binary additive divisor problem concerns estimates for the quantity
\begin{equation} \label{eqn1.2}
|\{ \vec{x} \in \mathbb{N}^4 : Q(\vec{x}) = h \ \text{and}  \ x_3 x_4 \leq N^2  \}| = \sum_{1 \leq n \leq N^2} d(n) d(n+h)
\end{equation}
for positive integers $h, N$, where $d(n) = \sum_{x,y \in \mathbb{N}}\1_{n = xy}$ is the divisor function. This is a very well-studied problem, see \cite{DI1982, Meu2001, Mo1994} and the references therein, in part due to its close connections to the fourth moment of the Riemann zeta function on the half line \cite{HB1979}. 

In this paper, we are interested in analysing the quantity
\[ T(h,N) = |\{ \vec{x} \in \mathbb{Z}^4 : Q(\vec{x}) = h \ \text{and} \ \vec{x} \in [-N,N]^4 \}| = \sum_{|n| \leq N^2} d'(n) d'(n+h), \]
where $d'(n) = \sum_{x,y \in [-N,N]} \1_{n = xy}$ can be interpreted as a restricted version of the divisor function. This specific setting has seen significant recent activity, see \cite{Af2024, CM2025, CM2025b, DHP2025, GG2024, martin}. Since this is a problem concerning a quadratic form in four variables with  sharp cut-offs instead of smooth weights and with $h \in [1,2N^2]$ being quite arbitrary,  the circle method does not seem to give an asymptotic formula for $T(h,N)$ with strong error terms and uniformity in $h$, see \cite{Nie2010}. Nevertheless, the circle method heuristic would predict that the main term contribution should be of the order $N^2$ with the leading coefficient being a product of local densities and the real density. Thus, for every prime $p$, we define the $p$-adic density 
\[ \sigma_p(h) = \lim_{n \to \infty}p^{-3n} \{ \vec{x} \in (\mathbb{Z}/p^n \mathbb{Z})^4 : Q(\vec{x}) \equiv h \mmod{p^n} \}. \]
 Similarly, we also define the real density $\sigma_{\infty}(h/N^2)$ and the singular series $\mathfrak{S}_h$ to be
\[ \sigma_{\infty}(h/N^2) = \lim_{\eta \to 0^+} (2\eta)^{-1} \int_{[-1,1]^4} \1_{|Q(\vec{x}) - h/N^2| < \eta } \d \vec{x} \ \ \text{and}  \ \ \mathfrak{S}_h = \prod_{p \ \text{prime}} \sigma_p(h). \]
With this in hand, the circle method heuristic would predict that $T(h,N) \sim \sigma_{\infty}(h/N^2) \mathfrak{S}_h N^2$.

Despite the aforementioned issues with sharp cut-offs and $h \in [1,2N^2]$ being arbitrary, we are able to prove the above heuristic quantitatively by using a combination of combinatorial and analytic arguments in physical space along with bounds for Kloosterman sums.

\begin{theorem} \label{thm1.1}
    For any $h, N \in \mathbb{N}$, the quantities $\sigma_{\infty}(h/N^2)$ and $\mathfrak{S}_h$ exist, and 
    \[ T(h,N) = \sigma_{\infty}(h/N^2) \mathfrak{S}_h N^2 + O_{\eps}(N^{3/2 + \eps}).   \]
\end{theorem}

A natural point of comparison for this is its smoothened analogue, and so, let $w : \mathbb{R}^4 \to [0,1]$ be some smooth, compactly-supported function satisfying some further nice properties such as $w(\vec{0}) = 0$, see \cite{HB1996}. Here, a nice result of Heath-Brown \cite[Theorem 4]{HB1996} gives us
\begin{equation} \label{eqn1.3}
T_{w}(N^2,N) := \sum_{\vec{x} \in \mathbb{Z}^4} w(\vec{x}/N) \1_{Q(\vec{x}) = N^2} = \sigma_{\infty,w} \mathfrak{S}_h N^2 + O_{\eps, w}(N^{3/2 + \eps}), 
\end{equation}
where $\sigma_{\infty,w}$ is a smoothened version of $\sigma_{\infty}(1)$, see \cite[Theorem 3]{HB1996}. 
One can approximate the sharp cut-off with smooth weights and then use \eqref{eqn1.3}, but this makes the error terms extremely weak, see \cite[Theorem 3.6]{Nie2010}. Nevertheless, despite dealing with sharp cut-offs, our error term in Theorem \ref{thm1.1} matches in strength its analogue in \eqref{eqn1.3}. Combining Theorem \ref{thm1.1} with various properties of $\sigma_{\infty}(h/N^2)$ and $\mathfrak{G}_h$ gives the following corollary.

\begin{corollary} \label{cor1.2}
    Let $\lambda \in [0, 2)$. Then $0 < \sigma_{\infty}(\lambda) \ll 1$. Moreover, there exists  $C_{\lambda} >0$ such that for any $N \geq C_{\lambda}$ and $1 \leq h \leq 2N^2$, writing $\Delta = h - \lambda N^2$, we have that
    \[ T(h,N) = \sigma_{\infty}(\lambda) \frac{1}{\zeta(2)} \bigg(\sum_{d\mid h}\frac{1}{d}\bigg) N^2 + O_{\eps}(N^{3/2+ \eps}+ N^{\eps}|\Delta|). \]
\end{corollary}

This was previously known when $\lambda = 0$ with error terms $\ll_{\eps} N^{\eps}(N+ |h|)$, see \cite{CM2025b, DHP2025}. These works employ very different types of techniques which can not analyse the case when $\lambda >0$. In the special case when $h$ is a prime $p$ satisfying $1< p/N^2 < 2$, Martin--White--Yip \cite{martin} proved an asymptotic formula for $T(p,N)$; see also work of Dhanda--Haynes--Prasala \cite[Theorem 2]{DHP2025} for the case when $\{h,N\} = \{p, \lfloor p^{1/2} \rfloor\}$ or $\{h,N\} = \{p^2,p\}$ for some prime $p$.

Noting the analogous conjectured bounds in the setting of \eqref{eqn1.1} and \eqref{eqn1.2}, see \cite[\S1]{CM2025b}, as well as the results of \cite{CM2025b, DHP2025}, one might conjecture that the first error term in Corollary \ref{cor1.2} should be $O_{\eps}(N^{1 + \eps})$, thus exhibiting square-root cancellation. When $\lambda =1$, we can exploit an additional available symmetry via Ramanujan sums to bypass the application of Kloosterman sum bounds and confirm the above speculation. This is our main result.

\begin{theorem} \label{thm1.3}
    Let $h, N \in \mathbb{N}$ satisfy  $1 \leq h \leq 2N^2$. Writing $\Delta = h -N^2$, we have that
    \[ T(h,N) = \left(\frac{8}{\zeta(2)} - 4\right) \bigg(\sum_{d\mid h}\frac{1}{d}\bigg)N^2 + O_{\eps}(N^{1+ \eps}+ N^{\eps}|\Delta|).  \]
\end{theorem}

In order to highlight the novelty of the asymptotic formula given by Theorem \ref{thm1.3}, we will compare this to other known asymptotic results in related settings. Thus, for any $h, N \in \mathbb{N}$, define 
\[ E(h,N) = T(h,N) - (8\zeta(2)^{-1} - 4) N^2 (\sum_{d\mid h}1/d).\]
Then Theorem \ref{thm1.3} delivers the following corollary.

\begin{corollary} \label{cor1.4}
    Let $h, N \in \mathbb{N}$ satisfy $h = N^2 + O(N)$. Then $E(h,N) \ll_{\eps}  N^{1 + \eps}.$
\end{corollary}

When $h$ is taken to be some arbitrary power of some prime, this also confirms a speculation of Dhanda--Haynes--Prasala \cite{DHP2025}. The latter authors employed very different type of arguments, which in particular relied very crucially on divisibility properties of primes, to prove Corollary \ref{cor1.4} in the special case when $\{h,N\} = \{p, \lfloor p^{1/2} \rfloor\}$ or $\{h,N\} = \{p^2,p\}$ for some prime $p$. They further speculated that this should hold when $h$ is a higher power of some prime, but that this does not seem to follow from the types of arguments given in their proof. Thus Corollary \ref{cor1.4} confirms their speculation in a very general form.

We can compare Corollary \ref{cor1.4} to its smoothened analogue in \eqref{eqn1.3} which satisfies the much weaker upper bound
\[ E_{w}(h,N) = T_{w}(h,N) - \sigma_{\infty,w} \mathfrak{S}_h N^2 \ll_{\eps} N^{3/2 + \eps}. \]
We now consider the $\ell^2$ analogue of this as described in \eqref{eqn1.1}, that is, we define
\[ 
E_{\ell^2}(h',N) = |\{ \vec{x} \in \mathbb{Z}^4:   Q(\vec{x}) = h' \ \text{and}  \ x_1^2 + \dots + x_4^2 \leq N^2\}| -  6 N^2 \sum_{d|h'} d^{-1}\]
for all $1 \leq h' \leq N^2$. In this case, an application of the deep spectral theory of automorphic forms gives us, for any $1 \leq h' \leq N^2$, the bound
\[ E_{\ell^2}(h',N)  \ll h'^{1/3} N^{4/3} \sum_{d|h'} \frac{1}{d}; \]
see \cite[Theorem 12.4]{Iw2002}. On the other hand, when $h = N^2 + O(N)$, this upper bound matches the size of the main term, and thus, is significantly larger than our upper bound in Corollary \ref{cor1.4}. For $h = N^2 + O(N)$, the only known upper bound of the form $E_{\ell^2}(h,N) = o(N^2 \sum_{d|h}1/d)$ follows from a general class of results proven by Oh \cite{Oh2004} using techniques from dynamics, although the latter do not seem to given any quantitative power saving upper bound. Next, we consider the binary additive divisor problem described in \eqref{eqn1.2}, and so, we define
\[ E_{\rm div}(h',N) = |\{ \vec{x} \in \mathbb{N}^4 : Q(\vec{x}) = h' \ \text{and}  \ x_3 x_4 \leq N^2  \}| - M_{\rm div}(h',N) \]
for all $h' \in \mathbb{N}$, where $M_{\rm div}(h',N)$ is the main term described by Motohashi in \cite[Theorem 1]{Mo1994}. When $h = N^2 + O(N)$, we know that $N^2 \ll M_{\rm div}(h,N) \ll_{\eps} N^{2 + \eps}$. Using spectral theory of automorphic forms, Motohashi \cite{Mo1994} proved that 
\[ E_{\rm div}(h,N) \ll_{\eps} N^{4/3 + \eps}, \]
see also work of Meurman \cite{Meu2001} on this topic. While this compares favourably to our error term in Theorem \ref{thm1.1}, this is still significantly larger than our upper bound in Corollary \ref{cor1.4}. 

Returning to our asymptotic formula described in Theorem \ref{thm1.1}, we note that we can  explicitly calculate $\sigma_{\infty}(\lambda)$ and $\mathfrak{S}_h$. In particular, letting ${\rm Li}_2$ denote the \emph{dilogarithm function} (see \eqref{eqn5.4}), we obtain the following estimates.

\begin{proposition}\label{prop1.5}
    Let $h \in \mathbb{N}$ and let $\lambda >0$. Then
    \begin{equation}\label{eqn1.4}
    \fS_h= \zeta(2)^{-1}\sum_{d\mid h}\frac{1}{d} \ \ \text{and} \ \  \sigma_{\infty}(\lambda) = 4J(\lambda) + 8K(\lambda),
    \end{equation}
    where $J(\lambda)$ and $K(\lambda)$ are defined by \eqref{eqn3.7} and \eqref{eqn4.4} respectively. Moreover, if $\lambda \in (0,1)$, then
    \begin{equation*}
        \sigma_\infty(\lambda) = 16 - 8(1-\lambda)\log(1-\lambda) + 4\lambda(\Li_2(\lambda) - 2(1+\zeta(2)) - 2\lambda\log\lambda + 2(1-\lambda)\log^2\lambda),
    \end{equation*}
    whilst if $\lambda \in (1,2)$, then
    \begin{align*}
    \sigma_{\infty}(\lambda)=
        8 + 8(\lambda-1)(\log(\lambda - 1) - 1) + 4\lambda(\zeta(2) - \log^2 \lambda - 2\Li_2(1/\lambda)).
    \end{align*}
    Finally, we have $\sigma_\infty(1) = 8-4\zeta(2)$, and $0<\sigma_\infty(\lambda)\leqslant 48$ for all $\lambda\in[0,2)$.
\end{proposition}

\subsection*{Proof ideas}
We will now give an outline of our proof ideas. 
Suppose $ax - by = h$ for some $a,x,b,y \in [-N,N] \cap \mathbb{Z}$. After removing the contribution of the solutions satisfying $axby = 0$, we divide our proof into analysing two terms. The first of these counts, up to some symmetry factor, the \emph{additive-type solutions}, that is, solutions to the equation $ax + by = h$, with $a,x,b,y \in [N]$, where $[N] := \{1,2,\dots, N\}$. Similarly, the second term counts the \emph{difference-type solutions} which count $a,x,b,y \in [N]$ such that $ax - by = h$.

Let us focus on the additive-type solutions, and so, after dividing out by $d:= (x,y)$, which also must divide $h$, we get that these equal
\[ \sum_{d|h, d \leq N}  \sum_{\substack{ 1 \leq u, v \leq N/d , \\ (u,v) = 1}} \sum_{a,b \in [N]} \1_{au + bv = h/d} .\]
Estimating the innermost sum is equivalent to counting points in an arithmetic progression with common difference $u$. Indeed, $b \equiv (h/d) \overline{v} \mmod{u}$, where $\overline{v}$ is the inverse of $v$ in $(\mathbb{Z}/u \mathbb{Z})^{\times}$. Moreover, the conditions $a,b \in [N]$ translate to constraints on the starting and end points $S_{u,v,h/d}$ and $T_{u,v,h/d}$ of the arithmetic progression which contains $b$. We write the number of admissible choices of $b$ as 
\[ u^{-1}(T_{u,v,h/d} - S_{u,v,h/d} ) \ \1_{T_{u,v,h/d} \geq S_{u,v,h/d}}   + E_{u,v,h/d} \ \1_{T_{u,v,h/d} \geq S_{u,v,h/d}} ,\]
where the first term can be interpreted as a main term contribution and $E_{u,v,h/d}$ is the sum of three error terms. The first of these error terms is simply a characteristic function which enforces a congruence condition on $v$ modulo $u$, and the other two are of the form $\psi((r_1 + r_2 \overline{x})/y)$, where $r_1, r_2$ are non-zero integers depending on $h/d$ and $N$, and $\{x,y\} = \{u,v\}$, and $\psi$ is the \emph{sawtooth function} defined in \eqref{eqn2.1}. 

We first analyse the contribution of the main term $u^{-1}(T_{u,v,h/d} - S_{u,v,h/d} )$. We begin by applying M\"{o}bius inversion to rewrite the factor $\1_{(u,v) =1}$ as $\sum_{k|(u,v)}\mu(k)$, thus introducing a further auxiliary sum. Our next idea here is to turn the discrete sum over $u,v$ into it's continuous version. In this endeavour, we note that the terms $S_{u,v,h/d}$ and $T_{u,v,h/d}$ can take two possible values each, and thus, we have to do a case-by-case analysis and prove that the discrete-to-continuous approximation in each case costs at most $O_{\eps}(N^{1 + \eps})$. After applying a further change of variables along with various further analytic manoeuvres, we able to show that the additive main term is approximately
\[ N^2 \sum_{\substack{d|h \\ d \leqslant N}} \frac{1}{d} \sum_{k\leqslant N/d}\frac{\mu(k)}{k^2} \int_{(0,1)^2} (\alpha \beta)^{-1}  \mu((0, \alpha)\cap (\lambda - (0, \beta ))) \d \alpha \d \beta , \]
where $\mu$ denotes the Lebesgue measure on $\mathbb{R}$, see Lemmas \ref{lem3.1} and \ref{lem3.2}. We now show that the above double integral is precisely the real density corresponding to the additive-type solutions. Moreover, one can massage the double arithmetic sums to obtain the singular series. One can proceed similarly in the case of difference-type solutions. Thus, we obtain the main terms for Theorems \ref{thm1.1} and \ref{thm1.3}.

What remains to estimate is the contribution of the error terms $E_{u,v,h/d}$. The congruence error terms can be bounded by $O_{\eps}(N^{1 + \eps})$ via standard elementary number theoretic lemmas. The analysis of sawtooth error terms requires a lot more work and is precisely the reason why we get the $O_{\eps}(N^{3/2 + \eps})$ error term in the conclusion of Theorem \ref{thm1.1}. We begin our analysis by proving that for any choice of $T_{u,v,h/d}$ and $S_{u,v,h/d}$, the condition that $T_{u,v,h/d} \geq S_{u,v,h/d}$ is equivalent to the vector $(u,v)$ lying inside a convex subset of $[1,N/d]\times[1,N/d]$. In particular, for every fixed $u \in [1,N/d]$, the variable $v$ varies in some interval $I_u$ of integers such that $|I_u| \leq N/d$. We can now use the truncated Fourier expansion of the sawtooth function along with some analytic and elementary number theoretic arguments to reduce the problem to estimating sums of the form
\begin{equation} \label{eqn1.5}
 \sum_{d|h, d \leq N} \sum_{0 < |s| \leq N^{100}} \frac{1}{s}  \sum_{1 \leq u \leq N/d} | \sum_{\substack{v \in I_u,  \\ (u,v) = 1}} e(s r \overline{v}/u)| 
\end{equation}
where $r$ is some fixed, non-zero integer. Using the upper bound $|I_u| \leq N/d$ along with upper bounds on incomplete Kloosterman sums, we can prove that this is bounded above by 
\begin{equation*}
 \sum_{d|h, d \leq N} \frac{N}{d}\sum_{0 < |s| \leq N^{100}} \frac{1}{s} \sum_{1 \leq u \leq N/d} \frac{(sr, u)^{1/2}}{u^{1/2}}.
 \end{equation*} 
Cauchy's inequality and standard elementary number theory lemmas allow us to bound the above by $O_{\eps}(N^{3/2 + \eps})$. This finishes our proof outline of Theorem \ref{thm1.1}.

For the purposes of Theorem \ref{thm1.3}, we have to be more careful in our analysis. We first note that since $h = N^2 + \Delta$, the number of difference-type solutions is at most 
\[ \ll_{\eps} N^{\eps} ( 1+ \Delta/N)^2 \ll N^{\eps}  (1 + \Delta).\]
Thus, it suffices to estimate the additive-type solutions. We proceed exactly as in the proof of Theorem \ref{thm1.1}, until we reach \eqref{eqn1.5}. Here, upon utilising the fact that $h = N^2 + \Delta$, we are able to show that for all $1 \leq u \leq N/d$, the interval $I_u$ satisfies $|I_u| = u + O( |\Delta|/dN)$. Thus, upon incurring a further cost of the order $O_{\eps}(N^{\eps}|\Delta|)$, we can ensure that each $I_u$ is a complete interval modulo $u$. We can now use  properties of Ramanjuan sums along with the triangle inequality, see \eqref{eqn2.2},  to bound \eqref{eqn1.5}  by
\[ N^{\eps} \sum_{0 < |s| \leq Q} \frac{1}{s} \sum_{1 \leq u \leq N/d} (sr,u) . \]
 As before, we can estimate these via elementary methods, this time obtaining the $O_{\eps}(N^{1 + \eps})$ error term.

\subsection*{Outline} We use \S\ref{sec2} to record various preliminary lemmas from elementary and analytic number theory that we will require throughout our paper. We dedicate \S\ref{sec3} to analysing the expected main term contribution in additive-type equations, and we utilise \S\ref{sec4} to do the same for difference-type solutions. In \S\ref{sec5}, we study the singular series and the real density for both additive-type and difference-type solutions, thus securing the main terms in Theorems \ref{thm1.1} and \ref{thm1.3}. We conclude \S\ref{sec5} by presenting the proof of Proposition \ref{prop1.5}. We obtain our general error term bounds for Theorem \ref{thm1.1} in \S\ref{sec6}. We employ our results from the preceding sections to record the proofs of Theorem \ref{thm1.1} and Corollary \ref{cor1.2} in \S\ref{sec7}. Finally, in \S\ref{sec8}, we obtain our error term bounds for Theorem \ref{thm1.3}, which in turn combines with the ideas presented in \S\ref{sec3} to finish our proof of Theorem \ref{thm1.3}.

\subsection*{Notation} We employ Vinogradov notation, that is, we write $Y \ll_{z} X$, or equivalently $Y =O_z(X)$, to mean that $|Y| \leq C_z X$, where $C_z>0$ is some constant depending on the parameter $z$. Unless stated otherwise, whenever $\eps$ appears in any bound, it will mean that the bound holds for every $\eps >0$, though the implicit constant may depend on $\eps$. We denote the greatest common divisor of two integers $a$ and $b$ by $(a,b)$.  

\subsection*{Acknowledgements} We thank Sam Chow, V. Vinay Kumaraswamy, and Trevor Wooley for helpful comments. JC is supported by EPSRC through Joel Moreira's Frontier Research Guarantee grant, ref. \texttt{EP/Y014030/1}. AM is supported by a Leverhulme Early Career Fellowship \texttt{ECF-2025-148}.

\subsection*{Rights}

For the purpose of open access, the authors have applied a Creative Commons Attribution (CC-BY) licence to any Author Accepted Manuscript version arising from this submission.

\section{Preliminaries} \label{sec2}

We begin by recording the following elementary lemma, see \cite[Lemma 2.3]{CM2025b}.

\begin{lemma} \label{lem2.1}
    Let $m$ be a positive integer and let $M \geq 1$ be a real number. Then 
    \[ \sum_{1 \leq y \leq M} (y,m) \ll_{\varepsilon} M m^{\varepsilon}. \]
\end{lemma}

Given $x \in \mathbb{R}$, we denote $\lVert x \rVert = \min_{n \in \mathbb{Z}} |x - n|$ and $e(x) = e^{2 \pi i x}$. We also define 
\begin{equation} \label{eqn2.1}
    \psi(x) = x - \lfloor x \rfloor - 1/2, 
\end{equation}
where $\lfloor x \rfloor$ is the  greatest integer less than or equal to $x$.   The function $\psi : \mathbb{R} \to [-1/2,1/2)$ is referred to as the \emph{saw-tooth function}. We will require the following truncated Fourier expansion for $\psi(x)$, see for example, \cite[\S5]{Hoo1963}.

\begin{lemma} \label{lem2.2}
    Let $Q \geqslant 1$ be a real number. For any $\theta\in\R$, we have
    \begin{equation*}
        \psi(\theta) = \sum_{0<|h|\leqslant Q}\frac{e(\theta h)}{2\pi i h} + O\left(\min\left\lbrace 1, \frac{1}{Q\lVert \theta\rVert}\right\rbrace\right).
    \end{equation*}
\end{lemma}

Given positive integers $y,n$, we will need the following bounds for Ramanujan sums
\begin{equation} \label{eqn2.2}
|\sum_{ t \in (y/\mathbb{Z}y)^{\times}} e(nt/y)| = |\sum_{d| (n,y)} d \mu(y/n)| \leq \sum_{d|(n,y)} d,
\end{equation}
see \cite[Theorem 8.6]{Ap1976}. We will also need the following consequence of the Weil bounds on Kloosterman sums recorded by Hooley \cite[Lemma 2]{Hoo1957}.

\begin{lemma} \label{lem2.3}
    Let $a,b,u$ be integers with $u>0$. Then 
    \[ \bigg|\sum_{\substack{1 \leq  x < u, \\ (x,u) = 1}} e( (ax + b \overline{x})/u )\bigg| \leq d(u) (b,u)^{1/2} u^{1/2},\]
    where $d(u)$ denotes the number of divisors of $u$, and $\overline{x}$ denotes the inverse of $x$ in $(\mathbb{Z}/q\mathbb{Z})^{\times}$.
\end{lemma}

Applying the divisor bound, we see that the right-hand side is bounded by $O_{\eps}((b,c)^{1/2} c^{1/2 + \eps})$. This estimate can be used to derive the following classical bound on incomplete Kloosterman sums due to Hooley \cite[Lemma 3]{Hoo1963}. 

\begin{lemma} \label{lem2.4}
    Let $I$ be an interval, let $u,r$ be non-zero integers such that $u, v \leq N^{10}$. Then 
    \[ \sum_{v \in I, (v,u) =1} e( r\overline{v}/u) \ll N^{\eps} \frac{(r,u)^{1/2}}{u^{1/2}} (|I| + u \log u).  \]
\end{lemma}

For the sake of completeness, we record its proof below.

\begin{proof}[Proof of Lemma \ref{lem2.4}]
For every $z \in \Z/u\Z$, define $f_I(z) = |I \cap (z + u \cdot \Z)|$. Moreover, for any $f : \Z/u\Z \to \mathbb{C}$, define its Fourier transform $\hat{f}: \Z/u\Z \to \mathbb{C}$ as
\[ \hat{f}(\xi) = \sum_{x \in \Z/u\Z} f(x) e(x \xi/u).  \]
Applying orthogonality gives
\begin{equation} \label{eqn2.3}
f_I(x) = u^{-1}\sum_{\xi =0}^{u-1} \hat{f_I}(\xi) e(- \xi x/u).
\end{equation}
Moreover, for any $0 \leq \xi \leq u-1$, we see that
\[ \hat{f_I}(\xi) = \sum_{x \in \Z/u\Z} f_I(x) e(x \xi/u) = \sum_{y \in I} e(y \xi/u)  . \]
Since $I$ is an interval, we deduce that
\begin{equation} \label{eqn2.4}
|\hat{f_I}(\xi)| \ll \min \left\{  |I|, \n{\xi/u }^{-1} \right\}. 
\end{equation}

We apply \eqref{eqn2.3} to obtain
\begin{align*}
   \sum_{v \in I, (v,u) = 1} e(r \overline{v}/u) 
    = \sum_{x \in (\mathbb{Z}/u \mathbb{Z})^{\times}} f_I(\xi) e( r \overline{x}/u) = u^{-1}\sum_{x \in (\mathbb{Z}/u \mathbb{Z})^{\times}}  e(r \overline{x}/u) \sum_{\xi =0}^{u-1} \hat{f_I}(x) e(- \xi x/u) .
\end{align*}
Interchanging the order of sums, we find that the right-hand side is
\[  u^{-1} \sum_{\xi = 0}^{u-1} \hat{f_I}(\xi) \sum_{x \in (\mathbb{Z}/u \mathbb{Z})^{\times}} e( (r\overline{x} - \xi x)/u) \ll u^{-1} \sum_{\xi =0}^{u-1} |\hat{f_I}(\xi)| (r,u)^{1/2} u^{1/2 + \eps},   \]
where this inequality follows from applying the triangle inequality and Lemma \ref{lem2.3}. Thus, we conclude from \eqref{eqn2.4} that
\begin{align*}
    \sum_{v \in I, (v,u) = 1} e(r \overline{v}/u)  & \ll N^{\eps}\frac{(r,u)^{1/2}}{u^{1/2}} \sum_{\xi =0}^{u-1} |\hat{f_I}(\xi)| \ll N^{\eps} \frac{(r,u)^{1/2}}{u^{1/2}} ( |I| + \sum_{\xi =1}^{u-1} \n{ \xi/u}^{-1} ) \\
    & \ll N^{\eps} \frac{(r,u)^{1/2}}{u^{1/2}} ( |I| + u \log u ). \qedhere
\end{align*}
\end{proof}

\section{Additive-type solutions} \label{sec3}

We commence our investigation into solutions of the quadratic determinant equation by studying additive-type solutions. Suppose then that $a,b,x,y\in[N]$ satisfy
\begin{equation}\label{eqn3.1}
    ax + by = h.
\end{equation}
Observe that the greatest common divisor $d\in[N]$ of $x$ and $y$ must divide $h$ for this equation to be valid. Dividing both sides by $d$ leads to the equation
\begin{equation*}
    au + bv = h/d\qquad (a,b\in[N],\; u,v\in[N/d],\; \gcd(u,v) =1).
\end{equation*}
Our goal therefore is to compute the following sum:
\begin{equation*}
    \sum_{\substack{d|h \\ d \leqslant N}} \sum_{\substack{1 \leq u,v \leq N/d,\\ (u,v) = 1}} \sum_{1 \leq a,b \leq N} \1_{au + bv = h/d} .
\end{equation*}

Notice that \eqref{eqn3.1} implies $bv \equiv  (h/d) \mmod{u}$. Given coprime $u$ and $v$, let $b_0 \in [u]$ be the unique integer satisfying $b_0 v \equiv (h/d) \mmod{u}$ and let $a_0 = (h/d - b_0v)/u$. We then see that any $a,b \in \mathbb{Z}$ satisfying $au + bv = h/d$ must satisfy $b = b_0 + su$ and $a= a_0 - sv$ for some $s\in\Z$. Thus, the number of such $a,b\in[N]$ is equal to the total number of integers $s$ in the range
\begin{equation}\label{eqn3.2}
U := \max\left\{ \frac{1- b_0}{u}, \frac{a_0 - N}{v} \right\} \leq s \leq \min \left\{ \frac{N- b_0}{u}, \frac{a_0 - 1}{v}\right\} =:  V.
\end{equation}
The number of such $s$ is exactly 
\begin{equation*}
    (  \lfloor V \rfloor- \lfloor U \rfloor + \1_{\mathbb{Z}}(U) ) \1_{V \geq U} = (V-U)\1_{V \geq U} + \1_{\mathbb{Z}}(U)\1_{V \geq U} - ( \psi(V) - \psi(U))\1_{V \geq U},
\end{equation*}
where $\psi(x) := x - \lfloor x \rfloor - 1/2$ is the sawtooth function.

We have therefore reduced our problem to studying $\cM(h,N) + \cE(h,N)$, where 
\begin{equation}\label{eqn3.3}
    \cM(h,N):=\sum_{\substack{d|h \\ d \leqslant N}} \sum_{\substack{1 \leq u,v \leq N/d,\\ (u,v) = 1}}(V-U)\1_{V \geq U}
\end{equation}
is the \emph{main term} and
\begin{equation}\label{eqn3.4}
    \cE(h,N):=\sum_{\substack{d|h \\ d \leqslant N}} \sum_{\substack{1 \leq u,v \leq N/d,\\ (u,v) = 1}}(\1_{\mathbb{Z}}(U)\1_{V \geq U} - ( \psi(V) - \psi(U))\1_{V \geq U}),
\end{equation}
is the \emph{error term}, with $U$ and $V$ being as defined in \eqref{eqn3.2}. We postpone the analysis of $\cE(h,N)$ until \S\ref{sec6}. The remainder of this section is devoted to examining the main term $\cM(h,N)$. We proceed to show that this discrete sum can be approximated by an integral.\\

Observe that
\begin{equation*}
    V-U=\min\left\{\frac{N}{u} ,  \frac{h/d}{uv} - \frac{1}{v} \right\} -  \max\left\{ \frac{1}{u}, \frac{h/d}{uv}-\frac{N}{v} \right\}  = V'- U',
\end{equation*}
where $V' = V + b_0/u$ and $U' = U + b_0/u$. Since our expressions for $V'$ and $U'$ do not involve $a_0$ or $b_0$, we can extend the definitions of $V'$ and $U'$ to all positive real numbers. Explicitly, for all $\alpha,\beta\in(0,\infty)$, we define
\begin{equation*}
    V'=V'(\alpha,\beta,h/d,N) = \min\left\{\frac{N}{\alpha} ,  \frac{h/d}{\alpha\beta} - \frac{1}{\beta} \right\},\quad U'=U'(\alpha,\beta,h/d,N) :=\max\left\{ \frac{1}{\alpha}, \frac{h/d}{\alpha\beta}-\frac{N}{\beta} \right\},
\end{equation*}
and
\begin{equation*}
    F_N(\alpha,\beta,h/d) = \max\{0,V'(\alpha,\beta,h/d,N) - U'(\alpha,\beta,h/d,N)\}.
\end{equation*}
Before moving on, we record the trivial inequalities
\begin{equation}\label{eqn3.5}
   0\leqslant  F_N(\alpha,\beta,h/d) \leqslant \frac{N}{\max\{\alpha,\beta\}}.
\end{equation}

Our expression for the main term has therefore become
\begin{equation*}
     \cM(h,N)=\sum_{\substack{d|h \\ d \leqslant N}} \sum_{\substack{1 \leq u,v \leq N/d,\\ (u,v) = 1}} F_N(u,v,h/d) .
\end{equation*}
As we would like to replace the discrete sum of $F_N$ by an integral over real numbers, we need to remove the coprimality condition. This is achieved via M\"{o}bius inversion. Starting from the classical identity
\begin{equation*}
    \sum_{k\mid n}\mu(k) = \1_{n=1} \quad (n\in\N),
\end{equation*}
we have
\begin{align*}
    \cM(h,N)&=\sum_{\substack{d|h \\ d \leqslant N}} \sum_{1 \leq u,v \leq N/d} F_N(u,v,h/d)\sum_{k\mid (u,v)}\mu(k) =  \sum_{\substack{d|h \\ d \leqslant N}} \sum_{k\leqslant N/d}\mu(k)\sum_{1\leqslant u,v\leqslant N/dk}F_N(ku,kv,h/d)\\
    &= \sum_{\substack{d|h \\ d \leqslant N}} \sum_{k\leqslant N/d}\frac{\mu(k)}{k}\sum_{1\leqslant u,v\leqslant N/dk}F_N(u,v,h/dk).
\end{align*}

For fixed $d\mid h$ and $k\leqslant N/d$ with $d\leqslant N$, we now investigate the sum
\begin{equation*}
    \sum_{1\leqslant u,v\leqslant N/dk}F_N(u,v,h/dk).
\end{equation*}
As indicated previously, by extending the definition of $F_N$ to all real numbers, we can replace this discrete sum with an integral
\begin{equation*}
   S_N(h/dk)= \int_1^{N/dk} \int_1^{N/dk} F_N(\alpha, \beta, h/dk)  \d \beta \d \alpha .
\end{equation*}
Recalling the trivial bound \eqref{eqn3.5} for $F_N$, we see that
\begin{equation*}
    S_N(h/dk) = \int_1^{\lfloor N/dk \rfloor + 1} \int_1^{\lfloor N/dk \rfloor + 1} F_N(\alpha, \beta, h/dk)  \d \beta \d \alpha  + O(N\log N),
\end{equation*}
and so
\begin{align*}
    \bigg\lvert S_N(h/dk)-\sum_{1\leqslant u,v\leqslant N/dk}F_N(u,v,h/dk) \bigg\rvert & \\ \ll N\log N+\sum_{1 \leq u,v \leq N/dk} & \int_{u}^{u+1} \int_{v}^{v+1} |  F_N(u,v,h/dk) -   F_N(\alpha, \beta, h/dk) |  \d \alpha \d \beta.
\end{align*}
We now show that this sum of integrals is much smaller than the expected size of the main term.

\begin{lemma} \label{lem3.1}
    Let $N\in\N\setminus\{1\}$ and $h\in\Z$ with $0\leqslant h\leqslant 2N^2$. Let $d,k\in[N]$ with $d\mid h$ and $k\leqslant N/d$. We have
    \[  \sum_{1 \leq u,v \leq N/dk} \int_{u}^{u+1} \int_{v}^{v+1} |  F_N(u,v,h/dk) -   F_N(\alpha, \beta, h/dk) |  \d \alpha \d \beta   \ll N (\log N)^2.  \]
\end{lemma}

\begin{proof}
 Let $I = [h/dk - 2, h/dk + 2]$. Using \eqref{eqn3.5}, the contribution of the cases when either $u$ or $v$ lies in $I$ is
    \begin{align*} &  \ll  \sum_{u \in I} \sum_{1 \leq v \leq N/dk} \int_{u}^{u+1} \int_{v}^{v+1} |  F_N(u,v,h/dk) -   F_N(\alpha, \beta, h/dk) |  \d \alpha \d \beta \\
    & \ll N  \sum_{u \in I} \sum_{1 \leq v \leq N/dk}  \int_{u}^{u+1}  \int_{v}^{v+1} \left(\frac{1}{v} + \frac{1}{\beta} \right) \d \beta \d \alpha \ll  N  \sum_{u \in I} \sum_{1 \leq v \leq N/dk} \frac{1}{v}  \ll N \log N.
     \end{align*}
If neither $u$ nor $v$ lie in $I$, then we note that
\[\frac{N}{u} - \frac{h/dk}{uv} \leq \frac{ - 1}{v} \Leftrightarrow Nv  - h/dk \leq - u \Leftrightarrow v \leq h/(dkN) \]
and
\[  \frac{1}{u} - \frac{h/dk}{uv} \geq \frac{- N}{v}  \Leftrightarrow   Nu + v \geq h/(dk) \Leftrightarrow  u \geq h/(dkN)  .  \]

We now perform a case by case analysis. Our first case is when  $u, v \notin I$ satisfy $u,v \leq h/(dkN)$.  Then 
\[ F_N(u,v,h/dk) 
 = \max\left\{0, \frac{N}{u} + \frac{N}{v} - \frac{h/dk}{uv} \right\} = \max\left\{0, \frac{N(u+v - h/(dkN))}{uv} \right\}.\]
 For the maximum to be positive on the right-hand side, we can assume that $u + v \geq h/(dkN)$. Now, let $\alpha, \in [u,u+1)$ and $\beta \in [v, v+1)$ be real numbers. Then a standard Taylor expansion can be used to give the upper bound
\begin{align*}
|F_N(u,v,h/dk)  - F_N(\alpha, \beta, h/dk)| 
& \ll N \left(  \frac{h/(dkN)-v}{u^2 v}  + \frac{h/(dkN)-u}{v^2 u} \right) .
\end{align*}
Using the fact that $u+v \geq h/(dkN)$, we see that the right-hand side is bounded above by $O(N/(uv))$. Thus, we have
\begin{align*}
    \sum_{\substack{1 \leq u,v \leq h/(dkN), \\ u,v \notin I}} & \int_{u}^{u+1} \int_v^{v+1}| F_N(u,v, h/dk) -  F_N(\alpha,\beta,h/dk) | \  \d \alpha \d \beta \\
    & \ll  \sum_{\substack{1 \leq u,v \leq h/(dkN), \\ u,v \notin I}} \frac{N}{uv} \ll N (\log N)^2.
\end{align*}

We now proceed to our second case where $u,v \notin I$ are such that $u \leq h/(dkN)$ and $v > h/(dkN)$. In this case, we have
\[ F_N(u,v,h/dk) = \frac{N-1}{v} \]
and so, given  $\alpha, \in [u,u+1)$ and $\beta \in [v, v+1)$, we see that
\[ F_N(u,v,h/dk) - F_N(\alpha, \beta, h/dk) = (N-1)(v^{-1} - \beta^{-1}) \ll \frac{N}{v^2}. \]
Thus, 
\begin{align*}
\sum_{\substack{1 \leq u \leq h/(dkN) < v \leq N/dk, \\u,v \notin I}   } & \int_u^{u+1} \int_v^{v+1} | F_N(u,v, h/dk) - F_N(\alpha, \beta, h/dk) |  \d \alpha \d \beta  \\
& \ll  N \sum_{\substack{1 \leq u \leq h/(dkN) < v \leq N/dk, \\u,v \notin I}} \frac{1}{v^2} \ll N \sum_{h/(dkN) < v \leq N/dk} \frac{1}{v} \ll N \log N.
\end{align*}
A similar analysis yields the same bound for the third case where $u,v \notin I$ are such that $u > h/(dkN)$ and $v \leq  h/(dkN)$.

The final case to consider is when $u, v \notin I$ are such that $u,v > h/(dkN)$. Here, we find that
\[ F_N(u,v,h/dk)= \max\left\{0, \frac{h/(dk) - u - v}{uv} \right\}. \]
This quantity is positive only when $u + v < h/(dk)$. Thus, letting $u, v \notin I$ be integers satisfying $u, v > h/(dkN)$ and $u + v < h/(dk)$, and letting $\alpha \in [u,u+1)$ and $\beta \in [v,v+1)$, we may apply Taylor expansion to deduce that
\[ |F_N(u,v,h/dk) - F_N(\alpha, \beta, h/dk)| \ll \frac{h/dk - v}{u^2 v} + \frac{h/dk - u}{uv^2} \ll \frac{h/dk}{uv} \left( \frac{1}{u} + \frac{1}{v} \right) \ll  \frac{N}{uv},\]
As in the first case, we may now deduce that
\[  \sum_{\substack{h/(dkN) \leq u, v \leq N/dk, \\ u+ v < h/(dk), \\ u,v \notin I }} \int_{u}^{u+1} \int_v^{v+1} |F_N(u,v,h/dk) - F_N(\alpha, \beta, h/dk)| \d \alpha \d \beta \ll N (\log N)^2,  \]
and the proof is complete.
\end{proof}

Lemma \ref{lem3.1} therefore delivers the following expression for the main term:
\begin{equation}\label{eqn3.6}
    \cM(h,N)= \sum_{\substack{d|h \\ d \leqslant N}} \sum_{k\leqslant N/d}\frac{\mu(k)}{k}S_N(h/dk) + O_\eps(N^{1+\eps}).
\end{equation}
Our next task is to modify the integral $S_N(h/dk)$ to more closely align with the main term appearing in Theorem \ref{thm1.1}. Given $\lambda\in\R$, we define
\begin{equation}\label{eqn3.7}
     J(\lambda):=\int_{(0,1)^2} \frac{1}{\alpha \beta} \ \mu( (0, \beta] \cap (\lambda - (0, \alpha]  ) )\d \alpha \d \beta,
\end{equation}
where $\mu$ denotes the Lebesgue measure on the real line. It is immediate from the definition that $J(\lambda)=0$ for all $\lambda\in\R\setminus(0,2)$. Moreover, in view of the crude bound
    \begin{equation}\label{eqn3.8}
        \mu( (0, \beta] \cap (\lambda - (0, \alpha]  ) ) \leqslant\min\{\alpha,\beta,\sqrt{\alpha\beta}\} \qquad(0\leqslant\alpha,\beta\leqslant 1),
    \end{equation}
it is easy to verify that $J(\lambda)$ is well-defined and satisfies $0< J(\lambda)\leqslant 4$ for all $0<\lambda< 2$.

\begin{lemma} \label{lem3.2}
    Let $N,h,d,k$ be as in Lemma \ref{lem3.1}. Writing $h = \lambda N^2$, we have
    \[ S_N(h/dk) = \frac{N^2}{dk} J(\lambda)  + O(N(\log N)^2 ). \]
\end{lemma}
\begin{proof}
Recalling \eqref{eqn3.5} and applying a change of variables, we see that
\begin{align*}
S_N(h/dk) & = \frac{N}{dk} \int_{dk/N}^1 \int_{dk/N}^1 F_N( \alpha, \beta , h/N)  \d \beta \d \alpha \\
& = \frac{N}{dk} \int_{dk/N}^1 \int_{dk/N}^1 \max\left\lbrace 0,\min\left\{\frac{N}{\alpha} ,  \frac{h}{N\alpha\beta} - \frac{1}{\beta} \right\}-\max\left\{ \frac{1}{\alpha}, \frac{h}{N\alpha\beta}-\frac{N}{\beta} \right\}\right\rbrace  \d \beta \d \alpha \\
& = \frac{N^2}{dk} \int_{dk/N}^1 \int_{dk/N}^1 \frac{1}{\alpha \beta} \max \left\{  0, \min\left\{ \beta, \frac{h}{N^2}-\frac{\alpha}{N} \right\} - \max \left\{ \frac{\beta}{N} , \frac{h}{N^2}-\alpha\right\} \right\}  \d \beta \d \alpha.
\end{align*}
Now observe that for any $t,x_1,x_2,y_1,y_2\in\R$ with $x_1 \leqslant y_1$ and $x_2 \leqslant y_2$ we have
\[ \mu([x_1, y_1] \cap (t - [x_2, y_2]) ) = \max \{ 0, \min\{ y_1, t - x_2\} - \max\{ x_1, t - y_2 \} \} .\]
Recalling that $h=\lambda N^2$, we therefore infer
\[ S_N(h/dk) = \frac{N^2}{dk} \int_{dk/N}^1 \int_{dk/N}^1 \frac{1}{\alpha \beta} \ \mu \left( \left[ \frac{\beta}{N}, \beta \right] \cap \left(\lambda - \left[ \frac{\alpha}{N}, \alpha \right] \right) \right)  \d \beta \d \alpha.  \]
Furthermore, in view of the bound
\[ \frac{N^2}{dk} \int_{dk/N}^1 \int_{dk/N}^1  \frac{1}{\alpha N} \ d \beta d \alpha \ll \frac{N}{dk} ( \log N + \log (dk) ) \ll N \log N, \]
we may write
\begin{equation*}
    S_N(h/dk) = \frac{N^2}{dk} \int_{dk/N}^1 \int_{dk/N}^1 \frac{1}{\alpha \beta} \ \mu \left( \left[ 0, \beta \right] \cap \left(\lambda - \left[ 0, \alpha \right] \right) \right)  \d \beta \d \alpha +O(N\log N).
\end{equation*}
If we extend the range of both integrals from $[dk/N,1]$ to $(0,1)$, then
\begin{equation*}
    S_N(h/dk) = \frac{N^2}{dk} \int_{0}^1 \int_{0}^1 \frac{1}{\alpha \beta} \ \mu \left( \left[ 0, \beta \right] \cap \left(\lambda - \left[ 0, \alpha \right] \right) \right)  \d \beta \d \alpha +O(I_1+I_2+N\log N),
\end{equation*}
where
\begin{align*}
   I_1 &= \frac{N^2}{dk}\int_{dk/N}^1\int_0^{dk/N} \frac{1}{\alpha \beta} \ \mu( [0, \beta] \cap (\lambda - [0, \alpha] ) )\d \beta \d\alpha,\\
   I_2 &= \frac{N^2}{dk} \int_0^{dk/N} \int_{0}^1 \frac{1}{\alpha \beta} \ \mu( [0, \beta] \cap (\lambda - [0, \alpha] ) )\d \beta \d\alpha.
\end{align*}
Hence, to complete the proof, it only remains to show that $I_1, I_2=O(N(\log N)^2)$.

The crude bound \eqref{eqn3.8} is enough to establish
\begin{equation*}
    I_1\leqslant \frac{N^2}{dk}\int_{dk/N}^1 \frac{1}{\alpha}\int_0^{dk/N} 1 \d\beta\d\alpha =N\log(N/dk) \leqslant N\log N.
\end{equation*}
We now consider $I_2$. Recall that
\begin{equation*}
    \mu( [0, \beta] \cap (\lambda - [0, \alpha] ) ) = \mu([\max\{0,\lambda -\alpha\}, \min\{\beta,\lambda\}]).
\end{equation*}
Let $\ell = \min\{ \lambda,dk/N\}$. We may assume $\ell>0$, as otherwise $I_2=0$ and we are done. Thus, 
\begin{align*}
    \int_0^\ell \int_{0}^1 \frac{1}{\alpha \beta} \ \mu( [0, \beta] \cap (\lambda - [0, \alpha] ) )\d\beta \d \alpha  &= \int_0^\ell \frac{1}{\alpha}\left(\int_{\lambda -\alpha }^\lambda \frac{\alpha +\beta-\lambda}{\beta}\d\beta +\int_\lambda^1\frac{\alpha}{\beta}\d\beta \right)\d\alpha\\
    &= \ell(1-\log(\lambda)) + \lambda\int_0^{\ell/\lambda}\left(\frac{1}{t} - 1\right)\log(1-t)\d t.
\end{align*}
Observe that the function $t\mapsto (t^{-1}-1)\log(1-t)$ is continuous on $(0,1)$ and has finite limits as $t\to 0^+$ or $t\to 1^-$. Consequently, the integrand appearing in the final integral above is uniformly bounded by some absolute constant. Since we are in the case when $\lambda\neq 0$, we know that $N^{-2}\leqslant \lambda \leqslant 2$, and so
\begin{equation*}
    \int_0^\ell \int_{0}^1 \frac{1}{\alpha \beta} \ \mu( [0, \beta] \cap (\lambda - [0, \alpha] ) )\d\beta \d \alpha \ll \ell(1+|\log(\lambda)|)\ll \frac{dk}{N}\log(N).
\end{equation*}
If $\lambda\geqslant dk/N$, then $\ell=dk/N$ and this bound leads to the desired estimate $I_2=O(N\log N)$. If instead $0<\lambda<dk/N\leqslant 1$, then $\ell=\lambda\geqslant N^{-2}$ and the previous bound implies that
\begin{align*}
    I_2 + O(N\log N) &= \frac{N^2}{dk}\int_\lambda^{dk/N}\frac{1}{\alpha}\left(\int_0^{\lambda}\frac{\beta}{\beta}\d\beta+\int_\lambda^1\frac{\lambda}{\beta}\d\beta\right)\d\alpha\\
    &\ll \frac{\lambda N^2}{dk}(1+|\log(\lambda)|^2) \ll N(\log N)^2,
\end{align*}
as required. 
\end{proof}

By combining all of these observations, we have therefore established the following.

\begin{proposition}\label{prop3.3}
    Let $N,h\in\N$ with $1\leqslant h\leqslant 2N^2$, and write $h=\lambda N^2$. Then the number of solutions $(a,b,x,y)\in[N]^4$ to \eqref{eqn3.1} is equal to
    \begin{equation*}
        \frac {J(\lambda)}{\zeta(2)}N^2\sum_{d|h}\frac{1}{d} + \cE(h,N)+O_\eps(N^{1+\eps}),
    \end{equation*}
    where $J(\lambda)$ is given by \eqref{eqn3.7} and  $\cE(h,N)$ is defined in \eqref{eqn3.4} with $U$ and $V$ as in \eqref{eqn3.2}.
\end{proposition}
\begin{proof}
    In view of Lemma \ref{lem3.2}, the expression \eqref{eqn3.6} becomes
    \begin{equation*}
        \cM(h,N)= J(\lambda)N^2\sum_{\substack{d|h \\ d\leqslant N}}\frac{1}{d}\sum_{k\leqslant N/d}\frac{\mu(k)}{k^2} + O_\eps(N^{1+\eps}).
    \end{equation*}
    Note that
    \begin{equation*}
       0\leqslant  \sum_{d|h}\frac{1}{d} - \sum_{\substack{d|h \\ d\leqslant N}}\frac{1}{d} \leqslant \frac{1}{N}\sum_{d\mid h}1.
    \end{equation*}
    The result now follows from the divisor bound and the classical approximation
    \begin{equation*}
        \sum_{k\leqslant X}\frac{\mu(k)}{k^2} = \frac{1}{\zeta(2)} + O\left(\frac{1}{X}\right),
    \end{equation*}
    see for example \cite[\S 3.7]{Ap1976}.
\end{proof}

\section{Difference-type solutions}\label{sec4}

We now turn our attention to solutions of the difference equation
\begin{equation}\label{eqn4.1}
    ax - by = h\qquad (a,b,x,y\in[N]).
\end{equation}
Most of the steps of our analysis of this equation are the same as in our investigation of the additive equation \eqref{eqn3.1} in the previous section. For this reason, we omit many of the details and focus on the aspects that differ between the two arguments.

Suppose $d\in[N]$ is such that $d\mid h$. Given coprime $u,v\in[N/d]$, let $b_0 \in [u]$ be the unique integer such that $b_0 v \equiv -h/d \mmod{u}$, and let $a_0 = (h/d + b_0v)/u$. We then set
\begin{equation}\label{eqn4.2}
    U :=  \max\left\{ \frac{1- a_0}{v}, \frac{1 - b_0}{u} \right\};\quad  V:=\min\left\{\frac{N- a_0}{v} ,  \frac{N- b_0}{u}\right\}.
\end{equation}
As with the additive equation, we can express the number of solutions to \eqref{eqn4.1} in the form $\cM(h,N) + \cE(h,N)$, where these terms are as given in \eqref{eqn3.3} and \eqref{eqn3.4} but with $U$ and $V$ now defined by \eqref{eqn4.2}. For the rest of this section, we focus on establishing an asymptotic formula for the main term $\cM(h,N)$. The error term will be analysed in \S\ref{sec6}.

To remove the explicit dependence of the main term on $a_0$ and $b_0$, we set 
\[ V' = V + a_0/v = \min\left\{\frac{N}{v},\frac{N}{u} - \frac{h/d}{uv}\right\} \ \ \text{and}  \ \ U' = U + a_0/v = \max\left\{\frac{1}{v}, \frac{1}{u} - \frac{h/d}{uv} \right\} \]
In contrast with the case of additive-type solutions, it is significantly easier to describe the value of $(V'-U')\1_{V'\geqslant U'}$.  Indeed, we observe that
\begin{equation*}
    V' = \frac{N}{v} \Leftrightarrow  (u-v) \leq \frac{h}{dN};\qquad 
    U' = \frac{1}{v} \Leftrightarrow  (u-v) \geq \frac{h}{d},
\end{equation*}
and so
\begin{equation}\label{eqn4.3}
    uv(V'-U')\1_{V'\geqslant U'} =
    \begin{cases}
        Nv + h/d - u, \quad&\text{if }u  - v \geqslant h/d;\\
        (N-1)v, &\text{if }h/dN \leqslant u-v < h/d;\\
        Nu - h/d - v, &\text{if } v/N \leqslant u-h/dN < v;\\
        0, &\text{otherwise.}
    \end{cases}
\end{equation}
In view of the bounds
\begin{equation*}
        \sum_{\substack{d|h \\ d \leqslant N}} \sum_{\substack{1 \leq u,v \leq N/d,\\ (u,v) = 1}}\frac{h}{duv}\1_{u-v\geqslant h/d} \ll \sum_{\substack{d|h \\ d \leqslant N}} \sum_{\substack{1 \leq u,v \leq N/d,\\ (u,v) = 1}}\frac{1}{u} \ll_\eps N^{1+\eps},
\end{equation*}
we can remove the ``$h/d-u$'' and ``$-v$'' terms from the first and third cases respectively by incurring an additional $O_\eps(N^{1+\eps})$ error term in our final count. We can similarly replace ``$(N-1)v$'' with $Nv$ in the second case. Since for any fixed choice of $v\in[N/d]$ there is at most one choice of $u\in[N/d]$ for which $0\leqslant u -h/dN<v/N$, these bounds further show that we can replace ``$v/N$'' with $0$ in the lower bound condition in the third case of \eqref{eqn4.3}.

Combining all of the observations of the previous paragraph with M\"obius inversion allows us to rewrite our main term as
\begin{equation*}
    \cM(h,N) = \sum_{\substack{d|h \\ d \leqslant N}}\sum_{k\leqslant N/d}\frac{\mu(k)}{k} \sum_{\substack{1 \leq u,v \leq N/d,\\ (u,v) = 1}}G_N(u,v,h/dk) + O_\eps(N^{1+\eps}),
\end{equation*}
where, for all $\alpha,\beta\in(0,\infty)$, we define
\begin{equation*}
    G_N(\alpha,\beta,h/dk) := \1_{\alpha -\beta \geq \frac{h}{dkN}} \left( \frac{N}{\alpha} \right) + \1_{\frac{h}{dkN}\leq \alpha < \frac{h}{dkN}+ \beta}\left( \frac{N}{\beta} - \frac{h/dk}{\alpha\beta}\right).
\end{equation*}

Using a very similar proof strategy, we can prove the following analogue of Lemma \ref{lem3.1}:
    \begin{align*}
         \sum_{1 \leq u,v \leq N/dk}  \int_{u}^{u+1}\int_{v}^{v+1} | G_N(u,v, h/dk)  \ -G_N(\alpha,\beta, h/dk)| \d \beta \d \alpha \ll N(\log N)^2.
    \end{align*}
This shows that our main term is
\begin{equation*}
    \cM(h,N) = \sum_{\substack{d|h \\ d \leqslant N}}\sum_{k\leqslant N/d}\frac{\mu(k)}{k}T_N(h/dk) + O_\eps(N^{1+\eps}),
\end{equation*}
where
\begin{align*}
    T_N(h/dk) &:= \int_1^{N/dk} \int_1^{N/dk} G_N(\alpha, \beta, h/dk)  \d \beta \d \alpha \\
    & = \frac{N^2}{d k}  \int_{dk/N}^1 \int_{dk/N}^1 \left( \1_{\alpha - \beta \geq \frac{h}{N^2}} \left( \frac{1}{\alpha} \right) + \1_{\frac{h}{N^2}\leq \alpha < \frac{h}{N^2}+ \beta}\left( \frac{1}{\beta} - \frac{h/N^2}{\alpha \beta}\right) \right)\d \alpha\d \beta\\
    &= \frac{N^2}{dk}\int_{(dk/N,1)^2}\frac{1}{\alpha\beta}\mu([h/N^2,\beta+h/N^2]\cap [0,\alpha])\d\alpha\d\beta.
\end{align*}
Using analogous calculations as in the proof of Lemma \ref{lem3.2}, we find that
\begin{equation*}
    T_N(h/dk) = \frac{N^2}{dk}K(h/N^2) + O(N(\log N)^2),
\end{equation*}
where, for all $\lambda\in\R$,
\begin{equation}\label{eqn4.4}
    K(\lambda):= \int_{(0,1)^2}\frac{1}{\alpha\beta}\mu([\lambda,\beta+\lambda]\cap [0,\alpha])\d\alpha\d\beta.
\end{equation}
Together these observations deliver the following analogue of Proposition \ref{prop3.3}.
\begin{proposition}\label{prop4.1}
    Let $N,h\in\N$ with $1\leqslant h\leqslant 2N^2$, and write $h=\lambda N^2$. Then the number of solutions $(a,b,x,y)\in[N]^4$ to \eqref{eqn4.1} is equal to
    \begin{equation*}
        \frac {K(\lambda)}{\zeta(2)}N^2\sum_{d\mid h}\frac{1}{d} + \cE(h,N)+O_\eps(N^{1+\eps}),
    \end{equation*}
    where $K(\lambda)$ is given by \eqref{eqn4.4} and  $\cE(h,N)$ is defined in \eqref{eqn3.4} with $U$ and $V$ as in \eqref{eqn4.2}.
\end{proposition}

\section{The singular series and real density}\label{sec5}

In this section, we examine the integrals \eqref{eqn3.7} and \eqref{eqn4.4} appearing in our expressions for the number of additive-type and difference-type solutions respectively and compute their exact values. We also discuss how these integrals and the term $\zeta(2)^{-1}\sum_{d\mid h}(1/d)$ can be interpreted as the real density and the singular series respectively for their associated counting problems.

\subsection{The singular series}

We begin with the singular series. For each prime $p$ and $k\in\N$, let $N_{p,k}(h)$ denote the number of solutions to $x_1x_2 + x_3 x_4 \equiv h \mmod{p^k}$, where $x_1, \ldots, x_4 \in \Z/p^k \Z$. For each $n\in\N$, we also let $\nu_p(n)$ denote the largest non-negative integer $r$ such that $p^r$ divides $n$. With this in hand, we prove the following lemma.

\begin{lemma} \label{lem5.1}
    Let $p$ be prime. For any $m\in\Z\setminus\{0\}$ and any integer $k > 10 \nu_p(m)$, one has 
\begin{equation} \label{eqn5.1}
N_{p,k}(m) = p^{3k} \left(1 - \frac{1}{p^2}\right)\left(1 + \frac{1}{p} + \cdots + \frac{1}{p^{\nu_p(m)}}\right). 
\end{equation}
In particular,
\begin{equation*}
    \sigma_p(m):= \lim_{k\to\infty}p^{-3k}N_{p,k}(m) = \left(1 - \frac{1}{p^2}\right)\left(1 + \frac{1}{p} + \cdots + \frac{1}{p^{\nu_p(m)}}\right). 
\end{equation*}
\end{lemma}
\begin{proof}
    Given $x_1, x_2,x_3, x_4 \in \Z/p^k \Z$ satisfying $x_1 x_2 + x_3 x_4 \equiv m \mmod{p^k}$, we define $\eta = \min\{ \nu_p(x_1), \nu_p(x_3)\}$. Note that $0 \leq \eta \leq \nu_p(m)$. For any $0 \leq i \leq \nu_p(m)$, the number of pairs $(x_1, x_3)$ satisfying $\eta =i$ is precisely 
\[ p^{2k-2i} ( 2( 1- 1/p) - (1- 1/p)^2 ) = p^{2k - 2i}(1 - 1/p^2 ). \] Given any such pair $(x_1, x_3)$, say, with $\nu_p(x_1) = i$, we see that fixing $x_4$ then fixes $x_2 \mmod{p^{k-i}}$. Moreover each such fixed value of $x_2 \mmod{p^{k-i}}$ lifts to precisely $p^i$ distinct $x_2 \in \Z/p^k \Z$ satisfying $x_1 x_2 + x_3 x_4 \equiv m \mmod{p^k}$. Thus, the number of solutions $x_1,\dots, x_4$ with $\eta =i$ is $p^{3k-i}(1-1/p^2)$. Summing this over all $0 \leq i \leq \nu_p(m)$ delivers the expression in \eqref{eqn5.1}.
\end{proof}

Combining Lemma \ref{lem5.1} with the facts that 
\begin{equation}\label{eqn5.2}
    \prod_p\left( 1-\frac{1}{p^2}\right) = \frac{1}{\zeta(2)} \qquad \text{and} \qquad \prod_p \left(\sum_{j=0}^{\nu_p(h)}\frac{1}{p^j}\right) = \sum_{d\mid h}\frac{1}{d}
\end{equation}
immediately delivers the formula for the singular series in Proposition \ref{prop1.5}.

\subsection{Real density of additive-type solutions}

We now turn to the real density for \eqref{eqn3.1}. For each $\lambda\in\R$ and $0<\eta\leqslant 1$, let
\begin{equation*}
    J_{\eta}(\lambda) := \frac{1}{2\eta}\int_{(0,1)^4}\1_{|x_1x_4 + x_2x_3 - \lambda| < \eta} \d \mathbf{x}; \quad 
    \widetilde{J}_\eta(\lambda):=\int_{(\eta,1)^2} \frac{1}{\alpha \beta} \ \mu( (0, \beta] \cap (\lambda - (0, \alpha]  ) )\d \alpha \d \beta.
\end{equation*}
Our goal is to show that $J_\eta(\lambda)\to J(\lambda)$ as $\eta\to 0^+$. We accomplish this by first approximating $J_\eta(\lambda)$ by $\widetilde{J}_\eta(\lambda)$, and then approximating $J(\lambda)$ by $\widetilde{J}_\eta(\lambda)$. This latter approximation will also subsequently aid in our proof of Corollary \ref{cor1.2}. We first require the following lemma to handle error terms which arise in these calculations.

\begin{lemma}\label{lem5.2}
For all $0\leqslant \lambda\leqslant 2$ and $0<\eta\leqslant 1$, we have
    \begin{equation*}
       I(\lambda;\eta):= \int_{(0,1]^2} \1_{ \lambda - 2 \eta < xy < \lambda + 2\eta }\d x \d y = O(\eta\log(2/\eta)).
    \end{equation*}
\end{lemma}
\begin{proof}
    We may assume throughout that $\lambda-2\eta<1$, as otherwise $I(\lambda;\eta)=0$ and we are done.
    If we additionally have $\lambda + 2\eta\geqslant 1$, then $|1-(\lambda +2\eta)|\leqslant 4\eta$ and so
    \begin{equation*}
        I(\lambda;\eta) = \int_{(0,1]^2} \1_{xy > \lambda - 2\eta }\d x \d y \leqslant \int_{(\lambda - 2\eta,1]^2} \d x\d y \ll \eta^2.
    \end{equation*}

    Henceforth, we assume $\lambda + 2\eta<1$. First consider the case where $\lambda\leqslant 6\eta$. The desired estimate for $I(\lambda;\eta)$ now follows from the crude bound
    \begin{align*}
        I(\lambda;\eta)\leqslant \int_{(0,1]^2} \1_{xy < \lambda + 2\eta }\d x \d y = \int_0^1\int_0^{\min\{1,(\lambda+2\eta)/y\}}\d x \d y &= (\lambda + 2\eta)(1-\log (\lambda + 2\eta))\\
        & \leqslant 8\eta(1+\log(2/\eta)).
    \end{align*}
    Now suppose instead that $\lambda>6\eta$. By considering when the integrand of $I(\lambda;\eta)$ is non-zero, we find that
    \begin{align*}
        I(\lambda;\eta) &= \int_{\lambda - 2\eta}^1\int_{(\lambda -2\eta)/y}^{\min\{1,(\lambda + 2\eta)/y\}} \d x \d y = \int_{\lambda + 2\eta}^1 \frac{4\eta}{y}\d y + \int_{\lambda -2\eta}^{\lambda + 2\eta}\left(1 - \frac{\lambda - 2\eta}{y}\right)\d y\\
        & = 4\eta \log\left(\frac{1}{\lambda + 2\eta}\right) + 4\eta - (\lambda- 2\eta)\log\left(1+\frac{4\eta}{\lambda - 2\eta}\right)\\
        &= O(\eta\log(2/\eta)) -(\lambda- 2\eta)\log\left(1+\frac{4\eta}{\lambda - 2\eta}\right). 
    \end{align*}
    Invoking the Taylor expansion $\log(1+x) = x -O(x^2)$ for $0<x<1$ now delivers the final required bound
    \begin{equation*}
        (\lambda- 2\eta)\log\left(1+\frac{4\eta}{\lambda - 2\eta}\right) = 4\eta - O\left( \frac{\eta^2}{\lambda - 2\eta}\right) = O(\eta).\qedhere
    \end{equation*}
\end{proof}

\begin{lemma} \label{lem5.3}
    For all $\lambda\in\R$ and $0<\eta\leqslant 1$, we have
    \begin{equation*}
        J_\eta(\lambda) = \widetilde{J}_\eta(\lambda) + O(\eta(\log(2/\eta))^2)= J(\lambda) + O(\eta(\log(2/\eta))^2),
    \end{equation*}
    where $J(\lambda)$ is defined by \eqref{eqn3.7}.
    In particular, for all $\lambda\in\R$, we have $J_\eta(\lambda)\to J(\lambda)$ as $\eta\to 0^+$.
    \end{lemma}

    \begin{proof}
    We may assume that $0\leqslant\lambda\leqslant 2$, as otherwise $J_\eta(\lambda)=\widetilde J_\eta(\lambda) = J(\lambda)=0$ and we are done. 
    We begin by observing that
\begin{align*}
  J_{\eta}(\lambda) = &\; (2\eta)^{-1} \int_{(\eta,1)^2} \frac{1}{ab} \int_{0}^a\int_0^b \1_{\lambda - \eta < u + v < \lambda + \eta}\d u\d v\d a \d b  \\
  &+ O\left( \eta^{-1} \int_{0}^{\eta} \int_{[0,1]^2} \1_{ \lambda - 2\eta < by < \lambda + 2 \eta}\d b\d y\d a \d x\right).
\end{align*}
Performing a change of variables and invoking Lemma \ref{lem5.2} to bound the error term delivers
\begin{equation*}
    J_\eta(\lambda)= (2\eta)^{-1}\int_{(\eta,1)^2} \frac{1}{ab} \int_{\lambda-\eta}^{\lambda + \eta}  \mu( [0,b] \cap ( z- [0,a]) )\d z \d a \d b+ O(\eta\log(2/\eta)).
\end{equation*}
Using the crude estimate
\begin{equation}\label{eqn5.3}
    |\mu( [0,b] \cap ( z- [0,a]) )-\mu( [0,b] \cap ( z'- [0,a]) )|\leqslant|z-z'|,
\end{equation}
we deduce
\begin{align*}
    J_\eta(\lambda)+O(\eta\log(2/\eta)) &= \int_{(\eta,1)^2} \frac{1}{\alpha\beta} \mu( [0,\beta] \cap ( \lambda- [0,\alpha]) )\d\alpha \d\beta + O\left(\int_{(\eta,1)^2} \frac{\eta}{\alpha\beta}\d\alpha \d\beta\right)\\
    &= \widetilde{J}_\eta(\lambda) + O(\eta(\log(2/\eta))^2).
\end{align*}

It only remains to extend the range of integration from $(\eta,1]^2$ to $(0,1]^2$. Using \eqref{eqn3.8}, we obtain the following bounds:
\begin{align*}
    \int_{(0,\eta]^2} \frac{1}{\alpha\beta} \mu( [0,\beta] \cap ( \lambda- [0,\alpha]) )\d\alpha \d\beta \leqslant \int_{(0,\eta]^2} \frac{1}{\sqrt{\alpha\beta}} \d\alpha\d\beta \ll \eta,\\
    \int_{(0,\eta]}\int_{(\eta,1]}\frac{1}{\alpha\beta} \mu( [0,\beta] \cap ( \lambda- [0,\alpha]) )\d\alpha \d\beta\leqslant \eta\int_{(\eta,1]}\frac{1}{\alpha}\d\alpha  \ll \eta\log(2/\eta),\\
    \int_{(\eta,1]}\int_{(0,\eta]}\frac{1}{\alpha\beta} \mu( [0,\beta] \cap ( \lambda- [0,\alpha]) )\d\alpha \d\beta\leqslant \eta\int_{(\eta,1]}\frac{1}{\beta}\d\beta  \ll \eta\log(2/\eta).
\end{align*}
Combining these with the definition of $\widetilde{J}_\eta(\lambda)$ establishes the desired estimate.
    \end{proof}

Our final task is to provide an explicit formula for the real density $J(\lambda)$. 
The most technical aspect of this calculation is the appearance of the dilogarithm function
\begin{equation} \label{eqn5.4}
    \Li_2(x):=-\int_0^x\frac{\log(1-t)}{t}\d t = \sum_{n=1}^\infty\frac{x^n}{n^2}\qquad (x\in[-1,1]).
\end{equation}
For our purposes, we only consider the dilogarithm restricted to $[0,1]$ and we record all the properties we need in the following lemma.
\begin{lemma}\label{lem5.4}
    The dilogarithm function defined above is continuous on $[0,1]$ and has the special values
    \begin{equation*}
        \Li_2(0) = 0;\quad \Li_2(1/2) = \frac{1}{2}(\zeta(2) - (\log(2))^2);\quad \Li_2(1) = \zeta(2).
    \end{equation*}
    Moreover, we have the \emph{reflection formula}\footnote{Here we interpret $\log(0)\log(1) = 0$.}
    \begin{equation*}
        \Li_2(x) + \Li_2(1-x) = \zeta(2) - \log(x)\log(1-x).
    \end{equation*}
\end{lemma}
\begin{proof}
    All of these properties follow readily from the definition of the dilogarithm and the Taylor expansion of the logarithm; see \cite{dilog} for further details.
\end{proof}

\begin{lemma}\label{lem5.5}
    For each $\lambda\in\R$, let $J(\lambda)$ be defined as in \eqref{eqn3.7}. Then we have
    \begin{equation*}
        J(\lambda)=
        \begin{cases}
        \lambda(\log^2(\lambda) - 2\log(\lambda) + 2-\Li_2(\lambda)), \quad &\text{if }0<\lambda<1;\\
        2-\zeta(2), &\text{if }\lambda = 1;\\
        2 + 2(\lambda-1)(\log(\lambda-1) - 1)+\lambda(\zeta(2)-\log^2(\lambda) - 2\Li_2(1/\lambda)), &\text{if }1<\lambda<2; \\
            0, &\text{otherwise}.
        \end{cases}
    \end{equation*}
\end{lemma}
\begin{proof}
    As noted in the remarks following \eqref{eqn3.7}, it is immediate from the definition that $J(\lambda) = 0$ whenever $\lambda\leqslant 0$ or $\lambda\geqslant 2$. We therefore only consider $0<\lambda<2$.

    First, consider the case where $0<\lambda< 1$. We have
    \begin{equation*}
        J(\lambda) = \int_{(0,1)^2}\frac{1}{\alpha\beta}\mu((\max\{0,\lambda -\alpha\},\min\{\beta,\lambda\}))\d \alpha\d\beta = I_1 + I_2 + I_3,
    \end{equation*}
    where
        \begin{align*}
        I_1 = \int_\lambda^1\int_0^1\frac{\min\{\beta,\lambda\}}{\alpha\beta}\d\beta\d\alpha = \int_\lambda^1\frac{\lambda}{\alpha}\d\alpha +\lambda\left(\int_{\lambda}^1\frac{1}{\alpha}\d\alpha\right)^2= \lambda(\log(\lambda)-1)\log(\lambda),
    \end{align*}
        \begin{align*}
        I_2 =  \int_0^\lambda\int_\lambda^1\frac{\lambda - (\lambda-\alpha)}{\alpha\beta}\d \beta\d \alpha = -\lambda\log(\lambda),
    \end{align*}
    and
    \begin{align*}
        I_3 &= \int_0^\lambda\int_{\lambda-\alpha}^\lambda\frac{\alpha+\beta-\lambda}{\alpha\beta}\d\beta\d\alpha =\lambda - \int_0^\lambda\frac{(\lambda-\alpha)(\log(\lambda) - \log(\lambda-\alpha))}{\alpha}\d\alpha \\
        &=\lambda + \lambda\int_0^\lambda\frac{(1-t)\log(1-t)}{t}\d t = \lambda(2-\Li_2(\lambda)),
    \end{align*}
    Adding all of these integrals delivers the required formula for $0<\lambda<1$. In view of the fact that $t\log(t)\to 0$ as $t\to 0^+$, performing the same calculations with $\lambda = 1$ and invoking Lemma \ref{lem5.4} gives
    \begin{equation*}
        J(1) =2-\Li_2(1) = 2-\zeta(2).
    \end{equation*}

    Now consider the case where $1<\lambda<2$. We have
    \begin{align*}
        J(\lambda) &= \int_{\lambda - 1}^1\int_{\lambda - \alpha}^1\frac{\beta-(\lambda-\alpha)}{\alpha\beta}\d\beta\d\alpha = \int_{\lambda -1}^1\frac{(\lambda-\alpha)\log(\lambda-\alpha) - (\lambda - \alpha - 1)}{\alpha}\d\alpha\\
        &= 2 + 2(\lambda-1)(\log(\lambda-1) - 1) + \lambda\int_{1-\frac{1}{\lambda}}^{\frac{1}{\lambda}}\frac{\log(\lambda)+\log(1-t)}{t}\d t\\
        & = 2 + 2(\lambda-1)(\log(\lambda-1) - 1) -\lambda\log(\lambda)\log(\lambda-1) + \lambda(\Li_2(1-\lambda^{-1}) - \Li_2(\lambda^{-1})).
    \end{align*}
    Applying the reflection formula from Lemma \ref{lem5.4} allows us to write
    \begin{equation*}
       \Li_2(1-\lambda^{-1}) = \zeta(2) + \log(\lambda)\log(1-\lambda^{-1}) - \Li_2(\lambda^{-1}).
    \end{equation*}
    Inserting this into the above expression for $J(\lambda)$ finishes the proof.
\end{proof}

\subsection{Real density of difference-type solutions}

For all $\eta>0$, let
\[ K_{\eta}(\lambda) :=    (2\eta)^{-1} \int_{(0,1)^4} \1_{|ax - by - \lambda| < \eta} \d a \d b \d x \d y. \]
Following similar calculations as in the proof of Lemma \ref{lem5.3}, we find that
\begin{equation}\label{eqn5.5}
    K_\eta(\lambda) = \widetilde K_\eta(\lambda) + O(\eta(\log(2/\eta)^2)= K(\lambda) + O(\eta(\log(2/\eta)^2) \qquad(\lambda\in\R),
\end{equation}
where
\begin{equation*}
    \widetilde{K}_\eta(\lambda):=\int_{(\eta,1)^2}\frac{1}{\alpha\beta}\mu([\lambda,\beta+\lambda]\cap [0,\alpha])\d\alpha\d\beta.
\end{equation*}
Proceeding as in the previous subsection, we now establish the following formula for $K(\lambda)$.

\begin{lemma}\label{lem5.6}
    For each $\lambda\in\R$, let $K(\lambda)$ be defined as in \eqref{eqn4.4}. For all real $0<|\lambda|< 1$, we have
    \begin{equation*}
        K(\lambda)= \lambda(\Li_2(\lambda) - \zeta(2) + (1/2)\log^2(\lambda)) + (1-\lambda)(\lambda\log(\lambda) - \log(1-\lambda)+2).
    \end{equation*}
    We also have $K(0)=2$. If $|\lambda|\geqslant 1$, then $K(\lambda)=0$.
\end{lemma}
\begin{proof}
    It follows immediately from the expression for the integrand in the definition of $K(\lambda)$ that $K(\lambda) = 0$ whenever $|\lambda|\geqslant 1$. Moreover, by interchanging $\alpha$ and $\beta$ and appealing to the identity
    \begin{equation*}
        \mu([\lambda,\beta+\lambda]\cap [0,\alpha]) = \mu([0,\beta]\cap [-\lambda,\alpha-\lambda]),
    \end{equation*}
    we note that $K(\lambda) = K(-\lambda)$ for all $\lambda\in\R$.

    By direct computation, we see that
    \begin{equation*}
        K(0) = \int_{(0,1)^2}\frac{\min\{\alpha,\beta\}}{\alpha\beta}\d\alpha\d\beta =  1 +\int_0^1\int_\alpha^1\frac{1}{\beta}\d\beta\d\alpha = 2.
    \end{equation*}
    It only remains to consider $0<\lambda<1$. In this case, we have
    \begin{align*}
        K(\lambda) &= \int_{(0,1)^2}\frac{1}{\alpha\beta}\mu([\lambda,\min\{\alpha,\beta+\lambda\}]\d\beta\d\alpha = \int_\lambda^1 \frac{\alpha-\lambda}{\alpha}\d\alpha - \int_\lambda^1\frac{(\alpha-\lambda)\log(\alpha-\lambda)}{\alpha}\d\alpha\\
        &= 1+\lambda(\log(\lambda) - 1) - (1-\lambda)(\log(1-\lambda) - 1) + \lambda\int_\lambda^1\frac{\log(\alpha - \lambda)}{\alpha}\d\alpha.
    \end{align*}
    Recall from Lemma \ref{lem5.4} that $\Li_2(1) = \zeta(2)$. The change of variables $\alpha = \lambda/t$ therefore reveals that
    \begin{align*}
        \int_\lambda^1\frac{\log(\alpha - \lambda)}{\alpha}\d\alpha &= -\lambda\log(\lambda) - \int_\lambda^1\frac{\log(t)}{t}\d t+ \int_\lambda^1\frac{\log(1-t)}{t}\d t\\
        &=\frac{1}{2}\log^2(\lambda) - \lambda\log(\lambda) + \Li_2(\lambda) -\zeta(2).
    \end{align*}
Inserting this into our previous expression completes the proof.
\end{proof}
\subsection{Proof of Proposition \ref{prop1.5}}

We conclude this section by proving Proposition \ref{prop1.5}. 

\begin{proof}[Proof of Proposition \ref{prop1.5}]
As noted earlier, the formula for the singular series $\fS_h$ follows from Lemma \ref{lem5.1} and \eqref{eqn5.2}. 
In view of the identity
\begin{equation*}
    \int_{[-1,1]^4} \1_{|x_1x_4+ x_2x_3 - \lambda| < \eta} \d \mathbf{x} = 4\int_{(0,1)^4} \1_{|x_1x_4+ x_2x_3 - \lambda| < \eta} \d \mathbf{x} + 8\int_{(0,1)^4} \1_{|x_1x_4- x_2x_3 - \lambda| < \eta} \d \mathbf{x}
\end{equation*}
and the definitions of $\sigma_\infty(\lambda)$, $J(\lambda)$, and $K(\lambda)$, we deduce from Lemma \ref{lem5.3} and \eqref{eqn5.5} that $\sigma_\infty(\lambda) = 4J(\lambda) + 8K(\lambda)$. We also noted, in the remarks following \eqref{eqn3.7}, that $0<J(\lambda)\leqslant 4$ whenever $0<\lambda<2$, and $J(\lambda) = 0$ otherwise. A similar argument shows that $0\leqslant K(\lambda)\leqslant 4$ for all $\lambda\in\R$, and we know from Lemma \ref{lem5.6} that $K(0)=2$. Together these observations reveal that
\begin{equation*}
    0\leqslant \sigma_\infty(\lambda)=4J(\lambda) + 8K(\lambda)\leqslant 48
\end{equation*}
for all $\lambda\in\R$, with $\sigma_\infty(\lambda)>0$ whenever $0\leqslant\lambda<2$. The remaining expressions for the values of $\sigma_\infty(\lambda)$ follow immediately from Lemmas \ref{lem5.5} and \ref{lem5.6}. This concludes the proof of Proposition \ref{prop1.5}.
\end{proof}

\section{Error term analysis for Theorem \ref{thm1.1}}\label{sec6}

In this section, we analyse the error term 
\begin{equation} \label{eqn6.1}
    \cE(h,N):=\sum_{\substack{d|h \\ d \leqslant N}} \sum_{\substack{1 \leq u,v \leq N/d,\\ (u,v) = 1}}(\1_{\mathbb{Z}}(U)\1_{V \geq U} - ( \psi(V) - \psi(U))\1_{V \geq U}) 
\end{equation} 
described in \eqref{eqn3.4}, where $U$ and $V$ are given by either \eqref{eqn3.2} or \eqref{eqn4.2}. We treat both the additive-type and difference-type error terms simultaneously by considering arbitrary $U$ and $V$ of the form $(r_1 + \overline{v} r_2)/u$ and $(q_1 + \overline{u} q_2)/v$, where $\overline{v}$ is the multiplicative inverse of $v$ in $(\mathbb{Z}/u \mathbb{Z})^{\times}$ and $\overline{u}$ is the multiplicative inverse of $u$ in $(\mathbb{Z}/v \mathbb{Z})^{\times}$ and $r_1, r_2, q_1, q_2$ are suitably chosen integers.

We begin by bounding the contribution of the first term in \eqref{eqn6.1}.

\begin{lemma} \label{lem6.1}
We have
\begin{equation*}
\sum_{\substack{d|h \\ d \leqslant N}} \sum_{\substack{1 \leq u,v \leq N/d,\\ (u,v) = 1}} \1_{\mathbb{Z}}(U)\1_{V \geq U} \leq \sum_{\substack{d|h \\ d \leqslant N}} \sum_{\substack{1 \leq u,v \leq N/d,\\ (u,v) = 1}} \1_{\mathbb{Z}}(U) \ll_{\eps} N^{1 + \eps}. 
\end{equation*}
\end{lemma}
\begin{proof}
Note that for any value that $U$ may take as described in \eqref{eqn3.2}, the hypothesis that $U$ is an integer either prescribes a congruence condition of the shape $v \equiv r_1 \mmod{u}$ or $u \equiv r_2 \mmod{u}$, for some integers $r_1, r_2$ which are independent of $u, v$. Thus, without loss of generality, we have
\[ \sum_{\substack{d|h \\ d \leqslant N}} \sum_{\substack{1 \leq u,v \leq N/d,\\ (u,v) = 1}} \1_{\mathbb{Z}}(U) \leq \sum_{d|h} \sum_{1 \leq u \leq N/d} \sum_{\substack{ 1 \leq v \leq N/d, \\ v \equiv r_1 \mmod{u}}}  1,  \]
for some fixed $r_1 \in \mathbb{Z}$. The innermost sum on the right hand side admits a trivial upper bound of the shape $O(N/du)$. Thus,
\[ \sum_{\substack{d|h \\ d \leqslant N}} \sum_{\substack{1 \leq u,v \leq N/d,\\ (u,v) = 1}} \1_{\mathbb{Z}}(U) \ll \sum_{d|h, d\leq N} \sum_{1 \leq u \leq N/d} \frac{N}{du} \ll_{\eps} N^{1 + \eps}, \]
which concludes our proof of Lemma \ref{lem6.1}.
\end{proof}

It remains to bound the contribution of the second and third terms in \eqref{eqn6.1}. Our goal is to prove the following. 

\begin{lemma} \label{lem6.2}
We have
\begin{equation*}
\sum_{d|h, d \leq N} \sum_{\substack{1 \leq u, v \leq N/d, \\ (u,v) = 1}}\psi(V) \1_{V \geq U} \ll_{\eps} N^{3/2 + \eps}. 
\end{equation*}
\end{lemma}

The case when $\psi(V)$ is replaced by $\psi(U)$ can be analysed similarly. In order to prove Lemma \ref{lem6.2}, we require the following preliminary estimate. 

\begin{lemma} \label{lem6.3}
    For every $1 \leq u \leq N/d$, let $I_u$ be some interval such that $|I_u| \leq N/d$. Let $r_1, r_2 \leq N^{10}$ be some non-zero integers. Then 
    \[ \sum_{1 \leq u \leq N/d} \sum_{\substack{v \in I_u\\ (u,v) = 1}} \psi((r_1 + \overline{v} r_2)/u) \ll N^{\eps} (N/d)^{3/2}.  \]
\end{lemma}

\begin{proof}
    We start by employing the truncated Fourier expansion of $\psi$. Thus, applying Lemma \ref{lem2.2} with $Q = N^{100}$, we have
    \begin{align} \label{eqn6.2}
    \sum_{1 \leq u \leq N/d} \sum_{\substack{v \in I_u, \\ (v,u) = 1}} \psi((r_1 + \overline{v} r_2)/u) 
    = & \sum_{1 \leq u \leq N/d}  \sum_{\substack{v \in I_u, \\ (v,u) = 1}}  \sum_{0 < |s| \leq Q} \frac{e(s(r_1 + \overline{v}r_2)/u )}{2 \pi i s} \nonumber \\
    & +  \sum_{1 \leq u \leq N/d} \sum_{\substack{v \in I_u, \\ (v,u) = 1}} O \left(  \min\left\{1, \frac{1}{Q \n{ (r_1 + r_2 \overline{v})/u}}   \right\} \right).
    \end{align}
    The absolute value of the first term on the right-hand side is bounded above by
    \[ \sum_{0 < |s| \leq Q} \frac{1}{2 \pi s} \sum_{1 \leq u \leq N/d} \bigg| \sum_{\substack{v \in I_u, \\ (v,u) = 1}} e( s r_2 \overline{v}/u) \bigg| \ll N^{\eps} \sum_{0 < |s| \leq Q} \frac{1}{s} \sum_{1 \leq u \leq N/d} \frac{(s r_2,u)^{1/2}}{u^{1/2}} ( |I_u| + u \log u). \]
    Here, the latter inequality follows from Lemma \ref{lem2.4}. Note that $u, |I_u| \leq N/d$, and so, the above is 
    \[ \ll N^{\eps} (N/d) \sum_{0 < |s| \leq Q} \frac{1}{s} \sum_{1 \leq u \leq N/d} \frac{(sr_2, u)^{1/2}}{u^{1/2}}.  \]
    Now, applying Cauchy's inequality to the inner sum and using Lemma \ref{lem2.1}, we find that
    \[ \sum_{1 \leq u \leq N/d} \frac{(sr_2, u)^{1/2}}{u^{1/2}} \leq \bigg( \sum_{1 \leq u \leq N/d} (sr_2, u) \bigg)^{1/2} \bigg( \sum_{1 \leq u \leq N/d} \frac{1}{u} \bigg)^{1/2} \ll N^{\eps} (N/d)^{1/2}. \]
    Putting everything together, we see that the first term on the right-hand side in \eqref{eqn6.2} is 
    \[ \ll  N^{\eps} (N/d)^{3/2} \sum_{0 < |s| \leq Q} \frac{1}{s} \ll N^{\eps} (N/d)^{3/2}. \]

    We now analyse the second term on the right-hand side in \eqref{eqn6.2}. We begin by noting that for any $t \in \mathbb{Z}$, we have $\n{t/u} \geq 1/u$ unless $u \mid t$. Thus,
    \begin{align} \label{eqn6.3}
        \sum_{1 \leq u \leq N/d} & \sum_{\substack{v \in I_u, \\ (v,u) = 1}}   \min\left\{1, \frac{1}{Q \n{ (r_1 + r_2 \overline{v})/u}}   \right\}   \ll \sum_{1 \leq u\leq N/d} \sum_{\substack{v \in I_u, \\ (u,v) = 1, \\ r_1 + r_2 \overline{v} \equiv 0 \mmod{u} }} 1 +  \sum_{1 \leq u \leq N/d} \sum_{\substack{v \in I_u, \\ (u,v) = 1}} \frac{u}{Q} \nonumber  \\
        & \ll  \sum_{\substack{ 1 \leq u \leq N/d, \\ u \nmid r_2}}  \frac{N/d}{u}  +    \sum_{\substack{ 1 \leq u \leq N/d, \\ u \mid r_2}}  N/d     +  \sum_{ 1 \leq u \leq N/d} \frac{N^2}{Q}  \ll N^{1 + \eps}/d. 
    \end{align}
    This completes the proof of Lemma \ref{lem6.3}.
\end{proof}

We now return to the proof of Lemma \ref{lem6.2}.

\begin{proof}[Proof of Lemma \ref{lem6.2}]
Notice that both $u,v$ always lie in the set $[1, N/d] \times [1,N/d]$. Let us consider the case of additive-type solutions first. The assumption that $U = (1 - b_0)/u$ holds is equivalent to the condition that $\alpha_1 u + \beta_1 v \leq \gamma_1$ for some choice of $\alpha_1, \beta_1, \gamma_1 \in \mathbb{R}$. A similar comment holds for fixing $V$ to equal either $(N- b_0)/u$ or $(a_0 - 1)/v$. Thus, fixing $U$ and $V$ to take specific values as recorded in \eqref{eqn3.2} and then subsequently fixing the condition that $V \geq U$ is equivalent to ensuring that the vector $(u,v)$ lies in a specific convex region $\cC$ inside  $[1, N/d] \times [1,N/d]$. This would mean that for any fixed $u \in  [1,N/d] \cap \mathbb{Z}$, the set $\cC \cap( \{u\} \times \mathbb{Z})$ is precisely of the form $\{u\} \times I_u$ where $I_u$ is some interval of integers. Moreover, since this is a subset of $[1, N/d] \times [1,N/d]$, we see that $|I_u| \leq N/d$. A similar comment holds for the case when we fix some $v \in [1,N/d] \cap \mathbb{Z}$ and consider all $u \in [1,N/d] \cap \mathbb{Z}$ which satisfy the above conditions. Now since there are two possibilities each for $U$ and $V$, all of which look like $(r + q\overline{x})/y$ where $\{x,y\} = \{ u, v\}$ and $r,q$ are non-zero integers lying in the interval $[-N^{10},N^{10}]$ such that $y \nmid q$, we see that
\begin{align*}
    \sum_{d|h, d \leq N} \sum_{\substack{1 \leq u, v \leq N/d, \\ (u,v) = 1}}\psi(V) \1_{V \geq U} \ll \sum_{d|h, d \leq N} \sum_{i=1}^4 \sum_{1 \leq x_i \leq N/d} \sum_{\substack{ y_i \in I_{x_i}, \\ (x_i,y_i) = 1}} \psi((r_i+q_i y_i)/x_i)  .
\end{align*}
We may now apply Lemma \ref{lem6.3} to deduce that
\[ \sum_{d|h, d \leq N} \sum_{\substack{1 \leq u, v \leq N/d, \\ (u,v) = 1}}\psi(V) \1_{V \geq U} \ll_{\eps} N^{3/2 + \eps} \sum_{d|h, d \leq N} 1/d^{3/2} \ll N^{3/2 + \eps}. \]
This is precisely the desired bound. 
\end{proof}

One can proceed mutatis mutandis when we study $\psi(V)$ with $\psi(U)$ and when we consider difference-type solutions.

\section{Proof of Theorem \ref{thm1.1} and Corollary \ref{cor1.2}} \label{sec7}

We now bring together all of the observations and results developed over the course of the paper to establish Theorem \ref{thm1.1} and Corollary \ref{cor1.2}. We begin with the proof of Theorem \ref{thm1.1}.

\begin{proof}[Proof of Theorem \ref{thm1.1}]
    Our goal is to obtain an asymptotic formula for the number $T(h,N)$ of integer solutions $(x_1,x_2,x_3,x_4)\in[-N,N]^4$ to the equation
    \begin{equation*}
        x_1x_4 - x_2x_3 = h.
    \end{equation*}
    By the divisor bound, there are at most $O_\eps(N^{1+\eps})$ solutions where at least one of the $x_i$ is equal to $0$. Consider then solutions where the $x_i$ are all non-zero. There are $16$ choices for the signs of the $x_i$. Since $h\geqslant 1$, the four choices which cause $x_1x_4 < 0<x_2x_3$ result in no solutions. The four choices of signs for which $x_2x_3<0<x_1x_4$ lead to counting solutions over $[N]^4$ to the additive equation \eqref{eqn3.1}. Similarly, the remaining eight sign choices reduce our problem to counting solutions over $[N]^4$ to the difference equation \eqref{eqn4.1}. Putting this all together, we have
    \begin{equation} \label{eqn7.1}
        T(h,N) = 4T_+(h,N) + 8T_-(h,N) + O_\eps(N^{1+\eps}),
    \end{equation}
    where
    \begin{equation*}
        T_\pm(h,N):=|\{ (x_1,x_2,x_3,x_4) \in \{1,\ldots,N\}^4 : x_1x_4\pm x_2x_3 =h\}|.
    \end{equation*}

    We evaluated the leading order of the quantities $T_\pm(h,N)$ in Propositions \ref{prop3.3} and \ref{prop4.1}, and so
    \begin{equation}\label{eqn7.2}
        T(h,N) = \frac{4J(h/N^2) + 8K(h/N^2)}{\zeta(2)}N^2 \sum_{d\mid h}\frac{1}{d} + 4\cE_+(h,N) + 8\cE_-(h,N) + O_\eps(N^{1+\eps}),
    \end{equation}
    where the error terms $\cE_+(h,N)$ and $\cE_-(h,N)$ are as defined in Propositions \ref{prop3.3} and \ref{prop4.1} respectively. Both of these error terms take the form \eqref{eqn6.1}. Invoking Lemmas \ref{lem6.1} and \ref{lem6.2}, we have
    \begin{equation*}
        4\cE_+(h,N) + 8\cE_-(h,N) \ll_\eps N^{\frac{3}{2} +\eps}.
    \end{equation*}
    Finally, Proposition \ref{prop1.5} allows us to rewrite the main term in \eqref{eqn7.2} into the form
    \begin{equation*}
        T(h,N) = \sigma_{\infty}(h/N^2) \mathfrak{S}_h N^2 + O_{\eps}(N^{3/2 + \eps}),
    \end{equation*}
    which finishes the proof of Theorem \ref{thm1.1}.
\end{proof}

We now turn our attention to Corollary \ref{cor1.2}. The key feature of this result is that the leading constant is bounded away from zero in terms of a given fixed $\lambda\in[0,2)$. To deduce this from Theorem \ref{thm1.1}, we need to understand how the real density $\sigma_\infty(\lambda)$ varies with $\lambda$. The following lemma is suitable for this purpose.

\begin{lemma}\label{lem7.1}
    Let $N\in\N$ and $\lambda\in[0,2)$. Let $\sigma_\infty(\lambda)$ be as defined in Theorem \ref{thm1.1}. For all $\Delta\in\Z\cap[-N^2,N^2]$, we have
    \begin{equation*}
        \sigma_\infty\left(\lambda+\frac{\Delta}{N^2}\right) = \sigma_\infty(\lambda) + O\left( \frac{|\Delta| \log^2(N)}{N^2}\right).
    \end{equation*}
\end{lemma}
 \begin{proof}
     The result is immediate if $\Delta=0$, so assume $\Delta\neq 0$ and set $\eta = |\Delta|/N^2$. Lemma \ref{lem5.3} shows that
     \begin{equation*}
         J(\lambda+\eta)  + O(\eta\log^2(2/\eta)) = \widetilde J_\eta(\lambda+\eta) = \int_{(\eta,1)^2} \frac{1}{\alpha \beta} \ \mu( (0, \beta] \cap (\lambda - (0, \alpha]  ) )\d \alpha \d \beta.
     \end{equation*}
     In view of the crude bound \eqref{eqn5.3} and the estimate
     \begin{equation*}
         \int_{(\eta,1)^2}\frac{\eta}{\alpha\beta}\d\alpha\d\beta \ll \eta\log^2(2/\eta),
     \end{equation*}
     the triangle inequality reveals that
     \begin{equation*}
         |J(\lambda+\eta) - J(\lambda)| \leqslant |J(\lambda+\eta) - \widetilde J_\eta(\lambda+\eta)| + |J(\lambda) - \widetilde J_\eta(\lambda)| +O(\eta\log^2(2/\eta)) \ll \eta\log^2(2/\eta).
     \end{equation*}
     A similar reasoning, combined with \eqref{eqn5.5}, leads to the bounds
     \begin{equation*}
         |J(\lambda-\eta) - J(\lambda)|,|K(\lambda+ \eta) - K(\lambda)|,|K(\lambda- \eta) - K(\lambda)|\ll \eta\log^2(2/\eta).
     \end{equation*}
     Recalling that $\eta\log^2(2/\eta) \ll \Delta N^{-2}\log^2(N)$ and $\sigma_\infty(\lambda) = 4J(\lambda) + 8K(\lambda)$ ends the proof.
 \end{proof}

We now present our proof of Corollary \ref{cor1.2}.

 \begin{proof}[Proof of Corollary \ref{cor1.2}]
     Fix some $\lambda\in[0,2)$. Proposition \ref{prop1.5} informs us that $0<\sigma_\infty(\lambda)\leqslant 48$. Writing $\Delta = h -\lambda N^2$, Theorem \ref{thm1.1} and \eqref{eqn1.4} show that
     \begin{equation*}
         T(h,N) = \sigma_{\infty}\left(\lambda + \frac{\Delta}{N^2}\right)\frac{N^2}{\zeta(2)}\sum_{d\mid h}\frac{1}{d} + O_{\eps}(N^{3/2 + \eps}).
     \end{equation*}
     The result now follows by applying Lemma \ref{lem7.1} to the above and invoking the divisor bound
     \begin{equation*}
         1\leqslant \sum_{d\mid h}\frac{1}{d}\ll_\eps N^\eps.\qedhere
     \end{equation*}
 \end{proof}

\section{Proof of Theorem \ref{thm1.3}}\label{sec8}

 In this section, we will present our proof of Theorem \ref{thm1.3}. Thus, let $h,N$ be positive integers and let $h = N^2 + \Delta$. We first consider the number of solutions to $ad - bc = h$ with $abcd = 0$ and $-N \leq a,b,c,d \leq N$. As in the proof of Theorem \ref{thm1.1}, we see that the number of such solutions is at most $O_{\eps}(N^{1 + \eps})$.  We will now show that the number of difference-type solutions, that is, the number of solutions satisfying $ad - bc = h$ with $1 \leq a,b,c,d \leq N$, is at most $O_{\eps}(N^{\eps}( 1 + |\Delta|)).$ Indeed,  suppose $ad - bc = h = N^2 + \Delta$ for some $a,b,c,d \in \{1,\dots, N\}$. This implies that 
 \[ a \geq ad/N \geq (N^2 + \Delta + bc)/N \geq N + \Delta/N. \]
 Since $a \in [N]$, we see that there are at most $O(1 + |\Delta|/N)$ many choices for $a$. Similarly, we have at most $O(1 + |\Delta|/N)$ many choices for $d$. Fixing such $a,d$, we see that $bc$ is fixed. Finally since we know that $bc \neq 0$, we can use the divisor estimate to fix $b,c$ in $O_{\eps}(N^{\eps})$ many ways. Thus, the number of difference solutions is at most
 \[ \ll_{\eps} N^{\eps} ( 1 + |\Delta|/N)^2 \ll N^{\eps} (1 + \Delta^2/N^2) \ll N^{\eps} (1 + |\Delta|), \]
 where the final step uses the fact that $|\Delta| = |h - N^2| \leq N^2$. 

 Hence, it suffices to estimate the number of additive-type solutions. We now follow our strategy described in \S\ref{sec3}. In particular, the main term here is 
\[ J(h/N^2) \zeta(2)^{-1}N^2\sum_{d|h}1/d. \]
Applying Lemma \ref{lem7.1}, we see that upon incurring a further cost of order $O_{\eps}(N^{\eps}|\Delta|)$, we may replace $J(h/N^2) = J(1 + \Delta/N^2)$ by $J(1)$. Substituting the value of $J(1)$ from Lemma \ref{lem5.5} into the above expression and noting \eqref{eqn7.1}, we find that our main term contribution to $T(h,N)$ is precisely
\[( 8\zeta(2)^{-1} - 4)N^2\sum_{d|h}1/d. \]

Thus, we now focus on the additive-type error terms. Recalling the definition of $U$ and $V$ in \eqref{eqn3.4}, we see that 
\[ U = \frac{a_0 - N}{v} \ \Leftrightarrow \ u \leq \frac{h}{dN} - \frac{v}{N} = \frac{N}{d} - \frac{v}{N} + \frac{\Delta}{dN} 
\]
and
\[ V = \frac{N- b_0}{u} \  \Leftrightarrow \ v \leq \frac{h}{dN} - \frac{u}{N} = \frac{N}{d} - \frac{u}{N} + \frac{\Delta}{dN} .  \]
Now note that $u, v \leq N/d$, and so, if $U \neq (a_0 - N)/v$, then we must have 
\[ N/d - v/N + \Delta/(dN) < u \leq N/d. \]
The contribution of these cases to the error term $\cE(h,N)$ in \eqref{eqn3.4} is at most
\begin{equation} \label{eqn8.1}
\leq \sum_{\substack{d|h \\ d \leqslant N}} \sum_{1 \leq v \leq N/d} \ \sum_{N/d - v/N + \Delta/(dN) < u \leq N/d} \  3 \ll  \sum_{\substack{d|h \\ d \leqslant N}} \sum_{1 \leq v \leq N/d} \left( \frac{\Delta}{dN} + \frac{v}{N} \right) \ll \Delta + N .
\end{equation}
We can similarly estimate the contribution of the terms when $V \neq (N - b_0)/u$. Thus, we may assume that $U = (a_0 - N)/v$ and $V = (N- b_0)/u$, in which case, $V \geq U$ is equivalent to the condition that $N(u + v) \geq h/d = N^2/d + \Delta/d$. Thus the contribution of these terms to $\cE(h,N)$ is 
\[ \sum_{\substack{d|h \\ d \leqslant N}} \sum_{\substack{1 \leq u,v \leq N/d,\\ (u,v) = 1, \\ u+ v \geq \frac{N}{d} + \frac{\Delta}{dN}}}(\1_{\mathbb{Z}}(U) - ( \psi(V) - \psi(U)) \1_{u \leq \frac{N}{d} - \frac{v}{N} + \frac{\Delta}{dN}} \1_{v \leq \frac{N}{d} - \frac{u}{N} + \frac{\Delta}{dN} }  . \]
As in the estimation of \eqref{eqn8.1}, we can remove the factors $\1_{u \leq \frac{N}{d} - \frac{v}{N} + \frac{\Delta}{dN}}$ and $\1_{v \leq \frac{N}{d} - \frac{u}{N} + \frac{\Delta}{dN} }$ from the above expression by incurring an error term of size $O(\Delta + N)$. Moreover, the contribution of the term $\1_{\mathbb{Z}}(U)$ to the above sum can be shown to be $O_{\eps}(N^{1 + \eps})$ exactly as in the proof of \eqref{eqn8.1}. Hence, it remains to show that
\begin{equation} \label{eqn8.2}
\sum_{\substack{d|h \\ d \leqslant N}} \sum_{\substack{1 \leq u,v \leq N/d,\\ (u,v) = 1, \\ u+ v \geq \frac{N}{d} + \frac{\Delta}{dN}}} \psi(V)   \ll N^{1 + \eps} \ \ \text{and} \ \ \sum_{\substack{d|h \\ d \leqslant N}} \sum_{\substack{1 \leq u,v \leq N/d,\\ (u,v) = 1, \\ u+ v \geq \frac{N}{d} + \frac{\Delta}{dN}}} \psi(U) \ll_{\eps} N^{1 + \eps} . 
\end{equation}
We will only record the proof of the first inequality, with the proof of the second bound following in the same vein with suitable modifications.

For our proof of the first inequality in \eqref{eqn8.2}, we will require the following lemma. 

\begin{lemma} \label{lem8.1}
   Let $\sigma >0$ be a parameter. For every $1 \leq u \leq N/d$, let $I_u$ be some interval such that $|I_u| = u + O(\sigma)$. Let $r_1, r_2 \leq N^{10}$ be some non-zero integers. Then 
    \[ \sum_{1 \leq u \leq N/d} \sum_{v \in I_u, (u,v) = 1} \psi((r_1 + \overline{v} r_2)/u) \ll  N^{1 + \eps}(1+ \sigma)/d .  \]
\end{lemma}

\begin{proof}
    We proceed as in Lemma \ref{lem6.3} and consider the truncated Fourier decomposition as in \eqref{eqn6.2}. Noting \eqref{eqn6.3}, we see that the second term on the right hand side in \eqref{eqn6.2} is already $O_{\eps}(N^{1+ \eps}/d)$ and so, it suffices to analyse the first term. Since $|I_u| = u + O(\sigma)$, we see that upon adding or removing at most $\sigma$ many terms, we can make $I_u$ a complete interval modulo $u$. This means that the first term in the right hand side of \eqref{eqn6.2} satisfies
\[ \bigg|  \sum_{1 \leq u \leq N/d}   \sum_{\substack{v \in I_u, \\ (v,u) = 1}}  \sum_{0 < |s| \leq Q} \frac{e(s(r_1 + \overline{v}r_2)/u )}{2 \pi i s} \bigg| 
     \ll \sum_{1 \leq u \leq N/d}  \sum_{0 < |s| \leq Q} \frac{1}{s} \Big( \Big|\sum_{\substack{1 \leq v' \leq u,\\ (v',u) = 1} } e(s v' r_2/u) \Big|+ \sigma \Big). \]
   The contribution of the term $\sigma$ in the inner sum to the entire sum is at most $\sigma N^{1 + \eps}/d$, whence we can analyse the first term in the inner sum above. Applying \eqref{eqn2.2} and Lemma \ref{lem2.1}, we get that
\begin{align*}
   \sum_{1 \leq u \leq N/d}  \sum_{0 < |s| \leq Q} \frac{1}{s} \Big|\sum_{\substack{1 \leq v' \leq u,\\ (v',u) = 1} } e(s v' r_2/u) \Big| & \ll \sum_{1 \leq u \leq N/d}  \sum_{0 < |s| \leq Q} \frac{1}{s} \sum_{r| (u, sr_2) } r  \ll_{\eps} N^{\eps}  \sum_{1 \leq u \leq N/d}  \sum_{0 < |s| \leq Q} \frac{(u, sr_2)}{s}   \\
    & = N^{\eps} \sum_{0 < |s| \leq Q} \frac{1}{s} \sum_{1 \leq u \leq N/d} (u, sr_2)  \ll_{\eps}  N^{1 + \eps}/d .
\end{align*}
Thus, we are finished with the proof of Lemma \ref{lem8.1}.
\end{proof}

We now return to the proof of the first inequality in \eqref{eqn8.2}. The left hand side of this inequality can be written as
\[ \sum_{\substack{d|h \\ d \leqslant N}} \sum_{1 \leq u \leq N/d} \ \sum_{\substack{ \frac{N}{d} + \frac{\Delta}{dN}  - u \leq v \leq \frac{N}{d},\\ (u,v) = 1} }  \psi( (N- b_0)/u )  . \]
Note that for each fixed $1 \leq u \leq N/d$, we have $v$ lying in an interval of length $u + O(|\Delta|/dN)$. Applying Lemma \ref{lem8.1}, we see that the above sum is at most 
\[ \ll_{\eps} \sum_{\substack{d|h \\ d \leqslant N}} \frac{N^{1+ \eps}}{d} \left( 1 + \frac{|\Delta|}{dN} \right)  \ll_{\eps} N^{\eps}(N + |\Delta|),  \]
which is precisely the desired bound.

\end{document}